\documentclass{amsart}

\usepackage{amsmath,amssymb,amsthm}

\usepackage[utf8]{inputenc}
\usepackage[T1]{fontenc}
\usepackage{hyperref}
\usepackage{cite}

\hypersetup{
breaklinks=true,
colorlinks=true,
linkcolor=blue,
citecolor=blue,
urlcolor=blue,
}

\newtheorem{thm}{Theorem}[section]

\newtheorem{lem}[thm]{Lemma}
\newtheorem{cor}[thm]{Corollary}
\theoremstyle{definition}
\newtheorem{defn}[thm]{Definition}

\theoremstyle{remark}

\numberwithin{equation}{section}

\newcommand{\R}{\mathbb{R}}  
\newcommand{\C}{\mathbb{C}} 

\newcommand{\Z}{\mathbb{Z}} 
\newcommand{\N}{\mathbb{N}} 
\newcommand{\E}{\mathcal{E}}
\newcommand{\Ell}{\mathcal{L}}

\newcommand{\RLZ}{\R / L \Z}
\newcommand{\RZ}{\R / \Z}
\newcommand{\Si}{\textnormal{Si}}

\begin{document}
	
\date{\today}

\title{On the analyticity of critical points of the Möbius energy}

\author{Simon Blatt}
\address{Fachbereich Mathematik, Universit\"at Salzburg, Hellbrunner Strasse 34, 5020 Salzburg, Austria}
\email{simon.blatt@sbg.ac.at}
\urladdr{http://blatt.sbg.ac.at/}

\author{Nicole Vorderobermeier}
\address{Fachbereich Mathematik, Universität Salzburg, Hellbrunner Strasse 34, 5020 Salzburg, Austria}
\email{nicole.vorderobermeier@sbg.ac.at}



\begin{abstract}
We prove that smooth critical points of the M\"obius energy $\E$ parametrized by arc-length are analytic. Together with the main result in \cite{BRS16} this implies that critical points of the M\"obius energy with merely bounded energy are not only $C^\infty$ but also analytic. Our proof is based on Cauchy's method of majorants and a decomposition of the gradient which already proved useful in the proof of the regularity results in \cite{BR13} and \cite{BRS16}. To best of the authors knowledge, this is the first analyticity result in the context of non-local differential equations.
\end{abstract}

\maketitle

\tableofcontents

\section{Introduction}

Knots are part of our everyday life, while tying our shoes, in art works or on a fisherwoman's boat. They have always played an important role in craft and trades, culture and arts. During the last two centuries, knot theory has arisen to a research topic of great interest and it has had a sustainable impact on mathematics, especially on the development of topology in the 19th century \cite{TvdG96}. Currently, the theory of knots appears in several branches of mathematics, such as calculus of variations, nonlinear and geometric analysis, and topology as well as in modern quantum physics, see for example in Kauffman’s essay \cite{K05}, and biochemistry (protein molecules, e.g. \cite{KS98}, or DNA, e.g. \cite{CKS98} and \cite{GM99}).

In mathematical terms, one can describe a knot as a continuous embedding of a circle into the three-dimensional Euclidean space. Alternatively, a knot can be represented by a closed curve $\gamma: \RZ \rightarrow \R^3$ which embeds a circle $\RZ$ into $\R^3$. Two knots belong to the same knot class if there exists a deformation from one knot to the other called ambient isotopy, which avoids "cutting and gluing" and any self-intersection. In order to determine a representative with a particular "nice" shape within a given knot class, so-called knot energies appeared three decades ago. A knot is considered as "nice" if it is as smooth, symmetric, and as barely entangled as possible, which may be reached by modeling a form of self-avoidance and subsequently finding a minimizer of the knot energy. 

Such a knot energy was first constructed by Fukuhara, who focused on polygonal knots \cite{F88}. Later, O’Hara investigated several potential energies of knots in \cite{OH91} and \cite{OH92}. Amongst others he defined for any closed, Lipschitz continuous and regular curve $\gamma: \RZ\rightarrow\R^n$ for any $n\in\N$, $n \geq 2$, the energy 
\begin{align*}
\E(\gamma) := \iint_{(\RLZ)^2} \left(\frac{1}{|\gamma(u)-\gamma(v) |^2} - \frac{1}{\mathcal{D}(\gamma(u),\gamma(v))^2} \right) |\gamma'(u)| |\gamma'(v)| \mathrm{d}u \mathrm{d}v
\end{align*}
which is called M\"obius energy due to its invariance under Möbius transformations proved in \cite[Theorem 2.1]{FHW94}. 
Here, $\mathcal{D}(\gamma(u),\gamma(v))$ denotes the intrinsic distance of $\gamma(u)$ and $\gamma(v)$ along the curve, i.e. 
\begin{align*}
\mathcal{D}(\gamma(x),\gamma(y)) := \min \{\Ell (\gamma|_{[x,y]}), \Ell(\gamma) - \Ell (\gamma|_{[x,y]}) \}
\end{align*}
for all $|x-y|\leq \frac{1}{2}$ where $\Ell (\gamma):= \int_0^1 |\dot{\gamma} (t)| \mathrm{d}t$ is the length of $\gamma$. 
It is easy to see that all values of $\E$ are non-negative, since the intrinsic distance of two points of the curve is always greater than the Euclidean distance. So $\E$ is indeed bounded from below. Furthermore, $\E$ is invariant under reparametrization of the curve due to the factor $|\gamma'(u)| |\gamma'(v)|$. O'Hara in \cite[Theorem 1.1]{OH94} and Freedman, He and Wang in \cite[Lemma 1.2]{FHW94} proved that $\E$ is self-repulsive. However, the Möbius energy is not tight due to \cite[Theorem 3.1]{OH94}; which means that $\E$ does not blow up for a sequence of small knots that pulls tight. Let us mention that Ishizeki and Nagasawa decomposed the Möbius energy into three parts each of which is again Möbius invariant \cite{IN14}  \cite{IN16}. One of these parts is equal to the cosine formula of Doyle and Schramm \cite{AS97}. 
For further information about the Möbius energy, other knot energies and their properties we refer to \cite{OH03} and \cite{SvdM13}.

Now let for the rest of this article $\gamma: \RZ \rightarrow\R^n$, $n \in\N$, be a simple, closed curve parametrized by arc-length if not stated differently. Here, "simple" means that $\gamma$ is injective on $[0,1)$. It was shown in \cite{B12} that such curves have finite M\"obius energy if and only if they belong to the fractional Sobolev space $H^{\frac 32}(\mathbb R / \mathbb Z, \mathbb R^n).$

He proved in \cite[Chapter 5]{H00} based on Freedman, He, and Wang \cite[Theorem 5.4]{FHW94} that any local minimizer $\gamma$ of $\E$ with respect to the $L^\infty$-topology and $n=3$ is smooth in sense that  $\gamma\in C^\infty (\RZ,\R^3)$. Reiter generalized this result to a larger class of O'Hara's knot energy in \cite{R12} and furnished a proof that any critical point $\gamma$ of $\E$ in  $H^2(\RZ,\R^n)$ is smooth. The latest result that will be useful for our further investigations comes from Blatt, Reiter and Schikorra:

\begin{thm}[\cite{BRS16}] \label{smooth}
	Any critical point $\gamma$ of $\E$ with $\gamma\in H^{\frac{3}{2}}(\RZ, \R^n)$ parametrized by arc-length, i.e. every $\gamma\in H^{\frac{3}{2}}(\RZ, \R^n)$ that satisfies 
	\begin{align*}
	\delta \E(\gamma;h) := \lim_{\tau \rightarrow 0, \tau \neq 0} \frac{F(u+\tau h)- F(u)}{\tau} = 0 \ \ \ \ \forall h\in C^\infty(\RZ,\R^n),
	\end{align*}
	belongs to $C^\infty (\RZ,\R^n)$.
\end{thm}

These results of course raise the question, whether a critical point of the Möbius energy in $H^{\frac{3}{2}}(\RZ, \R^n)$ is not only smooth, but also analytic. The main result of this article answers this question:

\begin{thm}\label{main}
	Let $\gamma :\RZ\rightarrow \R^n$ be a closed, simple, and by arc-length parametrized curve in $ H^{\frac{3}{2}} (\RZ,\R^n)$. If $\gamma$ is a critical point of the Möbius energy $\E$, then $\gamma$ is analytic.
\end{thm}

Together Theorem \ref{smooth} and the characterization of the energy spaces in \cite{B12} immediately gives the following corollary.

\begin{cor}\label{maincor}
	Let $\gamma :\RZ\rightarrow \R^n$ be a closed, simple, and by arc-length parametrized curve of finite Möbius energy. If $\gamma$ is a critical point of the Möbius energy $\E$, then $\gamma$ is analytic.
\end{cor}

There exist many results on the analyticity of solutions of classical analytic nonlinear elliptic partial differential equations. In 1904, Bernstein showed the analyticity of a solution of class $\mathcal{C}^3$ of nonlinear analytic elliptic equations of two independent variables in \cite{B04} by estimating higher order derivatives of solutions. Gevrey in \cite{G18} and Radó in \cite{R26} showed Bernstein's result by a similar strategy later. Also by the method of estimating higher order derivatives of solutions, Morrey and Nirenberg generalized the result to linear elliptic systems in \cite{MN57} in 1957 and moreover, Friedman to nonlinear elliptic systems of partial differential equations in \cite{F58} in 1958. Another method goes back to Lewy in 1929, who proved Bernstein's theorem by extension to the complex domain in \cite{L29}. By using this alternative method, Hopf in \cite{H32}, Petrovski in \cite{P39} and Morrey \cite{M58-1} and \cite{M58-2} generalized bit by bit Bernstein's theorem up to the analyticity of solutions of analytic nonlinear elliptic systems of partial differential equations as well. New ideas for proving the analyticity of solutions of analytic nonlinear elliptic equations were developed by Kato in \cite{K96} and Hashimoto in \cite{H06}.

To best of the authors knowledge, the result above is the first analyticity result in the context of non-local differential equations.

\subsubsection*{\textbf{Exposé of the present work.}} 

The main goal of this article is not only to give a rigorous proof of Theorem \ref{main} but also to give a proof which is as elementary as possible.
In section \ref{sec:preliminaries} we recall some basic definitions and properties of fractional Sobolev spaces and Fourier series. We also characterize analytic functions and state an ordinary differential equation (ODE)-version of the theorem of Cauchy-Kovalevskaya. We close the Preliminaries with recapitulating Faà di Bruno's formula. 
In section \ref{section:decomposition} we decompose the first variation of the Möbius energy into the orthogonal projection of a main part $Q$ and two remaining parts $R_1$ and $R_2$ of lower order. The main part $Q$ and its derivatives get estimated in section \ref{sec:maintermQ}. Since the orthogonal projection of $Q$ appears in the first variation of the Möbius energy, we estimate the tangential part of $Q$ using new estimates for a type of  bilinear Hilbert transform in section \ref{sec:bilinearHilberttransform}. In section \ref{sec:bilinearHilberttransform} and section \ref{sec:remaining parts} we will cut off the singularities of the singular integrals involved and derive uniform estimates. More precisely, in section \ref{sec:remaining parts} we rewrite  the orthogonal projection of the truncated remaining terms so that they may be expressed by integrals over analytic functions. 
 In section \ref{sec:mainpart} we will realize that since the estimates do not depend on the cutoff parameter, the same estimates hold for the actual functionals, more precisely for the orthogonal projections of $Q$, $R_1$ and $R_2$ and its derivatives. In the end we will use the  estimates to prove Theorem \ref{main} using the method of majorants.

\section{Preliminaries}\label{sec:preliminaries}

\subsection{Fractional Sobolev spaces}\label{sec:fractionalsobolevspaces}

Let us recall certain properties of fractional Sobolev spaces that are useful throughout the whole article. We assume some familiarity with the theory of weak derivatives and Fourier analysis. For a detailed introduction to Sobolev spaces we may refer to \cite{AF03}, \cite{E98} or \cite{T96}. 

Let $n\in \N$. We denote by $|\cdot | $ the real Euclidean norm $|x|:= \sqrt{\langle x, x\rangle_{\R^n}}$ on $\R^n$ and by $|\cdot |_{\C^n} $ the Euclidean norm $|z|_{\C^n}:= \sqrt{\langle z,z\rangle_{\C^n}}$ on $\C^n$.

In the sequel we work with closed curves on the periodic domain $\RZ$, i.e. we consider curves $f: \R \rightarrow \R^n$ which are periodic with period $1$, that means
\begin{align*}
f (x+1)= f(x) \quad  \forall x\in \R.	
\end{align*}

For $f\in L^2(\RZ, \C^n)$ and $k\in\Z$ the \emph{$k$-th Fourier coefficient of $f$} is defined as
\begin{align*}
\widehat{f}(k) = \int_0^1 f(x) e^{-2\pi i k x} \mathrm{d} x
\end{align*}
and the \emph{Fourier series of $f$} in $x\in \RZ$ is given by $\sum_{k\in\Z} \widehat{f}(k) e^{2\pi i k x}.$

The \emph{fractional Sobolev space} of order $s\geq 0$ (to be more precise: the Bessel potential space of order $s \geq 0$) is defined by
\begin{align*}
H^s(\RZ, \C^n) := \{f\in L^2 (\RZ, \C^n) \ | \ \| f \|_{H^s}:= \| f \|_{H^s(\RZ,\C^n)} := \sqrt{(f,f)_{H^s}} < \infty  \},
\end{align*}
equipped with the scalar product  
\begin{align*}
(f,g)_{H^s} := (f,g)_{H^s(\RZ,\C^n)} := \sum_{k\in\Z} (1+k^2)^s \langle \widehat{f}(k), \widehat{g}(k) \rangle_{\C^n}.
\end{align*}

We will use the embeddings $H^s(\RZ,\C^n)\subseteq H^t(\RZ,\C^n)$ for any $s>t$, $s,t\geq 0$ \cite[Proposition 2.1, Corollary 2.3]{DNPV12} as well as $H^s(\RZ,\C^n) \subseteq C(\RZ,\C^n)$ for any $s> \frac{1}{2}$ \cite[Theorem 8.2]{DNPV12}, which implies that the Fourier series of $f$ converges absolutely and uniformly to $f$ for any $f\in H^s(\RZ,\C)$, $s> \frac{1}{2}$. 

It is well-known that for $m \in \N_0$ the space $H^m$ agrees with the classical Sobolev space and the norm 
 $\|f\|_{W^m}:= \left(\sum_{\nu=0}^{m} \|\partial^\nu f\|^2_{L^2}\right)^{\frac{1}{2}}$ is equivalent to $\|\cdot\|_{H^m}$ (cf.  \cite[Lemma 1.2]{R09} or \cite[7.62]{AF03}).
Moreover, we note that for $m\in\N$ there exists a positive constant $C_m< \infty$ such that 
	\begin{align}\label{prop:Banachalgebra}
	\|fg\|_{H^m} \leq C_m \|f\|_{H^m} \|g\|_{H^m}
	\end{align}
for all $f,g \in H^m(\RZ,\R)$ by \cite[Theorem 4.39]{AF03}.

We will furthermore use the following simple consequence of the chain rule for Sobolev functions.

\begin{lem} \label{rem:compSobolevnorm}
	There exists a positive constant $C_c < \infty$ only depending on $d$ and $e$ such that 
	\begin{align}\label{formula:compH1}
	\|g\circ f\|_{H^1} \leq C_c \|g\|_{C^1} (1+ \|f\|_{H^1})
	\end{align} 
	for all $f \in H^1(\RZ ,\R^{d})$, $g \in C^1 (\R^{d} ,\R^e)$ and $e,d\in\N$. 
\end{lem}

\begin{proof}
	 This follows from
	\begin{align*}
	\|g\circ f\|_{W^1} & \leq \|g\circ f\|_{L^2} + \|(g\circ f){'}\|_{L^2} \\
	& \leq \|g\|_\infty + \|g'\|_\infty \|f'\|_{L^2} \\
	& \leq \|g\|_{C^1} (1+ \|f\|_{W^1})
	\end{align*}
	and by the equivalence of $W^1$- und $H^1$-norm as stated above.
\end{proof}

\subsection{Properties of analytic functions}

Let us recall some elementary facts about analytic functions.
More on the topic can be found  for example in  \cite[Chapter 6.4]{E98} and \cite[Chapter 1, D]{F95}. 
Let $n,m\in\N$ and $\Omega \subseteq \R^n $ be an open set throughout this section.

We will use the following well-known characterization of analytic functions, which can be deduced from \cite[Proposition 2.2.10]{KP02} by a standard covering argument. 
\begin{thm}\label{thm:Tayloranalytic}
	Let $f\in C^\infty (\Omega,\R^n)$. Then $f$ is analytic on $\Omega$ if and only if for every compact set $K \subset \Omega$
	there are constants $r_K>0$, $C_K < \infty$ such that
	\begin{align}\label{formula:partialf}
	\|\partial^\alpha f\|_{L^\infty(K)} \leq C_K \frac{|\alpha|!}{r_K^{|\alpha|}}
	\end{align}
	holds for every multiindex $\alpha \in \mathbb N_0^m$.
\end{thm} 

Combining this with standard embeddings for Sobolev spaces we get the following.

\begin{cor} \label{cor:TayloranalyticSob}
	Let $f\in C^\infty (\Omega,\R^n)$. Then $f$ is analytic on $\Omega$ if and only if for every compact set $K \subset \Omega$
	there are constants $r_K>0$, $C_K < \infty$ such that
	\begin{align}\label{formula:partialfH}
	\|\partial^\alpha f\|_{H^1} \leq C_K \frac{|\alpha|!}{r_K^{|\alpha|}}
	\end{align}
	holds for every multiindex $\alpha \in \mathbb N_0^m.$
\end{cor}

Let us mention a special case of the theorem of Cauchy-Kovalevskaya, which is originally an existence and uniqueness theorem for analytic nonlinear partial differential equations associated with Cauchy initial value problems. Proofs based on the \emph{method of majorants} which we will also use to prove the main result of this article can be found in \cite[Chapter 1, D]{F95} and \cite[Chapter 4.6.3, Theorem 2]{E98}.

\begin{thm}[Cauchy-Kovalevskaya - ODE case]\label{CK}
	Let $ \varepsilon > 0$. Suppose $g: \R^n \rightarrow \R^n$ is real analytic around $0$, and $f \in C^\infty(]-\varepsilon, \varepsilon[,\R^n)$, $f(x)= (f_1(x),\ldots,f_n(x))$, be a solution of the initial value problem
	\begin{align}
	\begin{cases}
	\dot{f}(x) & = g (f(x)) \textnormal{ for } x \in ]-\varepsilon,\varepsilon[\\
	f(0)& =0.
	\end{cases}
	\end{align}
	Then $f$ is real analytic around $0$.
\end{thm}

\subsection{Faà di Bruno's formula}

For two functions $f, g\in C^k (\R,\R)$ Faà di Bruno's formula tells us that
	\begin{align*}
	&\left(\frac{d}{dt}\right)^k g(f(t)) = \\
	&\sum_{\substack{m_1+2m_2+ \cdots km_k=k, \\ m_1,\ldots,m_k\in\N_0}} \frac{k!}{m_1!1!^{m_1}m_2!2!^{m_2}\ldots m_k!k!^{m_k}} g^{(m_1+\cdots +m_k)}(f(t))\prod_{j=1}^{k} (f^{(j)} (t))^{m_j}.
	\end{align*}
R. Mishkov generalized this formula to the multivariate case \cite{M00}.
\begin{lem}\label{rem:FaadiBruno}
	Let $U\subseteq\R$ be a neighborhood of $0$ and $f\in C^k(U, \R^n)$, $n\in\N_0$. Moreover, let $V\subseteq\R^n$ be a neighborhood of $f(0)$ such that $f(U)\subseteq V$ and $g\in C^k(V, \R)$. Then we have for any $k\in\N$ and $x\in U$
	\begin{align*}
	\partial^k g(f(x)) = \sum_0 \sum_1 \cdots \sum_k \frac{k!}{\prod_{i=1}^k (i!)^{r_i} \prod_{i=1}^k\prod_{j=1}^nq_{ij}! }  (\partial^\alpha g) (f(x)) \prod_{i=1}^k ( f_1^{(i)}(x))^{q_{i1}} \cdots ( f_n^{(i)}(x))^{q_{in}}
	\end{align*}
	where the sums are over all non-negative integer solutions of the following equations
	\begin{align*}
	& \sum_0 :  r_1 + 2r_2 + \cdots + k r_k = k, \\
	& \sum_1 : q_{11} + q_{12} + \cdots + q_{1n} = r_1 \\
	& \ \ \vdots  \\
	& \sum_k  : q_{k1} + q_{k2} + \cdots + q_{kn} = r_k
	\end{align*}
	and $\alpha := (\alpha_1,\ldots,\alpha_n)$ with $\alpha_j := q_{1j} + q_{2j} + \cdots + q_{kj}$ for $1\leq j \leq n$. 	
\end{lem}

In the following we will not need the precise form of Faà di Bruno's formula in the multivariat case but only the fact that there is a universal polynomial $p_k^{(n)}$ with non-negative coefficients independent of $f,g$ such that 
\begin{align}\label{universalpoly}
	\partial^k g(f(x)) = p_k^{(n)} (\{\partial_\alpha g\}_{|\alpha|\leq k}, \{\partial^j f_i\}_{i=1, \ldots, n, j=1, \ldots, k} ).
\end{align}
for all $f \in C ^k (\mathbb R, \mathbb R ^n)$ and $g \in C^k(\mathbb R ^n, \mathbb R).$ Furthermore, $p_k^{(n)}$ is one-homogeneous in the first entries.

\section{Decomposition of the first variation of $\E$}\label{section:decomposition}

In 1994, Freedman, He and Wang stated the Gâteaux differentiability of the Möbius energy $\E$ and additionally a formula of the first variation of $\E$ \cite[Lemma 6.1, Lemma 6.4.]{FHW94}. Reiter could prove that $\E$ is Fréchet differentiable and he computed its first variation at an injective regular curve $\eta\in H^2$ in direction $h\in H^2$ \cite[Chapter 2]{R12}. In order to work with the Möbius energy in a sophisticated way, He and Reiter used a kind of linearization of the first variation of $\E$ in \cite{H00} and \cite{R12}. Using these ideas and results, we want to determine a suitable decomposition of the first variation of the Möbius energy $\E$ at $\gamma$ which helps us proving the main result Theorem \ref{main}.  

Assume $n\in\N$ and let $\gamma\in C^\infty(\RZ,\R^n)$ be a simple, closed, and by arc-length parametrized curve. Recall that the orthogonal projection of $\R^n$ onto the normal vector plane to $\gamma$ at $\gamma(x)$  $P^\perp_{\dot{\gamma}(x)} : \R^n \rightarrow \R^n$ for every $g\in \R^n$ is given by 
\begin{align*}
(P^\perp_{\dot{\gamma}}g)(x) := P^\perp_{\dot{\gamma}(x)}(g) := g-\langle g,\dot{\gamma}(x) \rangle_{\R^n} \dot{\gamma}(x). 
\end{align*}
Then we get by \cite[Lemma 6.4.]{FHW94} and \cite[Theorem 2.24]{R12} the following

\begin{thm}
	The first variation of $\E$ at $\gamma \in H^2(\RZ,\R^n)$ in direction $h \in H^2 (\RZ,\R^n)$ may be written as
	\begin{align}\label{formula:deltaE}
	\delta \E(\gamma, h) = \int_{0}^{1} \langle (H\gamma)(x), h(x) \rangle_{\R^n} \mathrm{d} x
	\end{align}
	where
	\begin{align*}
	(H\gamma)(x) := 2 \lim_{\epsilon \downarrow 0} \int_{|w| \in [\epsilon, \frac{1}{2}]} \left(2 \frac{P^\perp_{\dot{\gamma}(x)}(\gamma(x+w)-\gamma(x))}{|\gamma(x+w)-\gamma(x)|^2} - \ddot{\gamma}(x) \right) \frac{\mathrm{d} w}{|\gamma(x+w)-\gamma(x)|^2}.
	\end{align*}
\end{thm}

Since $P^\perp_{\dot{\gamma}(x)}(\dot{\gamma}(x)) = 0$, $\langle \dot{\gamma}(x),\ddot{\gamma}(x) \rangle_{\R^n} = 0$ and by the linearity of $P^\perp_{\dot{\gamma}(x)}$ for all $x\in \RZ$, one sees that 
\begin{align}\label{formula:H}
(H\gamma) (x) = (P^\perp_{\dot{\gamma}}\widetilde{H} \gamma) (x) = P^\perp_{\dot{\gamma}(x)} ((\widetilde{H} \gamma) (x))
\end{align}
where 
\begin{align*}
(\widetilde{H} \gamma) (x) := 2 \lim_{\varepsilon \downarrow 0} \int_{|w| \in [\varepsilon, \frac{1}{2}]} \left(2 \frac{\gamma(x+w)-\gamma(x)-w\dot{\gamma}(x)}{|\gamma(x+w)-\gamma(x)|^2} - \ddot{\gamma}(x) \right) \frac{\mathrm{d} w}{|\gamma(x+w)-\gamma(x)|^2}.
\end{align*}

Now, motivated by \cite{H00} and \cite{R12} as explained above, simply by adding and subtracting the terms $\frac{4(\gamma(x+w)-\gamma(x)-w\dot{\gamma}(x))}{w^4}$ and $\frac{2\ddot{\gamma}(x)}{w^2}$, we may decompose $\widetilde{H}$ into three functionals $Q, R_1, R_2$ given point-wise for all $x\in [0,1]$ by 
\begin{align*}
	Q\gamma (x) = \lim_{\varepsilon \downarrow 0}Q^\varepsilon\gamma(x), \
	R_1 \gamma (x)= \lim_{\varepsilon \downarrow 0}R_1^\varepsilon\gamma(x), \
	R_2 \gamma (x) = \lim_{\varepsilon \downarrow 0}R_2^\varepsilon\gamma(x),
\end{align*}
where for any $0<\varepsilon\leq \frac 12$
\begin{align*}
(Q^\varepsilon\gamma)(x) &:= 2  \int_{|w| \in [\varepsilon, \frac{1}{2}]} \left(2 \frac{\gamma(x+w)-\gamma(x)-w\dot{\gamma}(x)}{w^2} - \ddot{\gamma}(x) \right) \frac{\mathrm{d} w}{w^2}, \\
(R_1^\varepsilon \gamma)(x)&:= 4 \int_{|w| \in [\varepsilon, \frac{1}{2}]} \left( \frac{1}{|\gamma(x+w)-\gamma(x)|^4} - \frac{1}{w^4} \right) (\gamma(x+w)-\gamma(x)-w\dot{\gamma}(x)) \mathrm{d} w ,\\
(R_2^\varepsilon \gamma)(x)&:= -2  \int_{|w| \in [\varepsilon, \frac{1}{2}]} \left( \frac{1}{|\gamma(x+w)-\gamma(x)|^2} - \frac{1}{w^2} \right) \ddot{\gamma}(x) \mathrm{d} w,
\end{align*}
such that
\begin{align}\label{formula:Htilde}
(\widetilde{H}\gamma)(x) = (Q\gamma)(x) + (R_1 \gamma)(x) + (R_2 \gamma)(x)
\end{align}
for all $x\in \RZ$.\\
In the upcoming section we will see that $Q$ contains the highest order part of the first variation of $\E$. We will see that the tangential part of $Q$ is the most difficult part to get under control, whereas $R_1$ and $R_2$ are of lower order and can be estimated more elementary.

\section{Estimate of the main term $Q$ }\label{sec:maintermQ}

Let
$$
 \lambda_n := \frac 2 3 \int_0^{n\pi} \frac 1 t \left( 1-  \frac t {n\pi}\right)^3 \sin(t) \mathrm{d}t.
$$
The following theorem is an immediate consequence of \cite[Lemma 2.3]{H00}.

\begin{thm}\label{thm:QFourier}
	For every $f \in C^\infty(\RZ, \R^n)$ the expression $Qf$ is a $C^\infty$ function whose Fourier coefficients are given by
	\begin{align*}
	\widehat{Qf}(k) = \frac {\pi^3}2 \lambda_k |k|^3 \hat{f} (k) && \forall k\in \Z.
	\end{align*}
	The constants $\lambda_k$ satisfy $0< \lambda_k = \frac \pi 3 + O(\frac 1 k)$ as $k \rightarrow \infty$.
\end{thm}
\begin{proof}
Using Taylor's expansion up to third order we get  
\begin{align}\label{eq:Qcont}
\begin{split}
Q\gamma (x) & = \lim_{\varepsilon\downarrow 0} 2  \int_{|w| \in [\varepsilon, \frac{1}{2}]}  \frac{\int_0^1 (1-t)^2 \dddot{\gamma}(x+tw) \mathrm{d}t }{w} \mathrm{d} w \\
& = \lim_{\varepsilon\downarrow 0} 2  \int_{|w| \in [\varepsilon, \frac{1}{2}]}  \frac{\int_0^1 (1-t)^2 \left(\dddot{\gamma}(x+tw) - \dddot{\gamma}(x)\right) \mathrm{d}t }{w} \mathrm{d} w ,
\end{split}
\end{align}
from which one immediately gets that $Q$ maps $C^{3,\alpha}$ to $L^\infty$, $0 < \alpha \leq 1$. As $Q$ is linear and hence $\partial^l Q \gamma (x) = Q\partial^l \gamma(x)$, $Q$ also maps $C^{l+3,\alpha}$ to $C^{l-1,1}$, $l\in\N_0$. This proves the first claim. 
For $f,g \in H^{\frac 32}(\mathbb R / \mathbb Z, \mathbb R ^n)$ we define the bilinear functional
\begin{equation*}
\widetilde Q (f,g) := \lim_{\varepsilon  \downarrow 0 } \int_{\mathbb R / \mathbb Z} \int_{|w| \in [\varepsilon, \frac 1 2]} \left( \frac {\langle f(x+w)-f(x),g(x+w)-g(x)\rangle }{w^4} - \frac {\langle f'(x) , g'(x) \rangle}{w^2} \right) \mathrm{d}w \mathrm{d}x.
\end{equation*}
Using partial integration and a discrete version thereof, we see that
\begin{align}\label{formula:Qtildecalc}
\begin{split}
\widetilde Q (f,g) & = \lim_{\varepsilon  \downarrow 0 } \int_{\mathbb R / \mathbb Z} \int_{|w| \in [\varepsilon, \frac 1 2]} \left( \frac {\langle f(x+w)-f(x),g(x+w)-g(x)\rangle }{w^4} - \frac {\langle f'(x) , g'(x) \rangle}{w^2} \right) \mathrm{d}w \mathrm{d}x \\
& = \lim_{\varepsilon  \downarrow 0 } \int_{\mathbb R / \mathbb Z} \int_{|w| \in [\varepsilon, \frac 1 2]} \left( -2 \frac {\langle f(x+w)-f(x),g(x)\rangle }{w^4} - \frac {\langle f'(x) , g'(x) \rangle}{w^2} \right) \mathrm{d}w \mathrm{d}x \\
& = \lim_{\varepsilon \downarrow 0} \int_{\mathbb R / \mathbb Z} \langle Q_{\varepsilon} f, g \rangle \mathrm{d}x 
= \int_{\mathbb R / \mathbb Z} \langle Qf ,g \rangle \mathrm{d}x  
\end{split}
\end{align}
since $Q^{\varepsilon} f$ converges to $Qf$ in $L^\infty(\RZ,\R^n)$ by (\ref{eq:Qcont}). 

From \cite[Lemma 2.3]{H00} we know that
\begin{align*}
 \widetilde Q(f,g) = \frac {\pi^3}2 \sum_{k \in \mathbb Z} \lambda_k |k|^3 \hat{f} (k) \hat {g}(k)
\end{align*}
and Plancherel's identity yields
\begin{align*}
  \int_{\mathbb R / \mathbb Z} \langle Qf ,g \rangle dx  = \sum_{k \in \mathbb Z} \widehat{Qf}(k) \widehat{g}(k).
\end{align*}
Thus comparing the Fourier coefficients in (\ref{formula:Qtildecalc}) gives
$\widehat{Qf}(k) = \frac {\pi^3}2 \lambda_k|k|^3 \hat{f} (k) $ for all $ k\in \Z.$

\end{proof}

All we need in the following is
\begin{cor} [of Theorem \ref{thm:QFourier}]\label{kor:Ql}
	Let $l\in\N_0$ and $m\geq 0$. There exists a positive constant $ \widetilde{C} < \infty $, which is independent of $l$, such that 
	\begin{align*}
	\|\partial^{l+3}\gamma \|_{H^m} \leq \widetilde{C} \|\partial^{l} Q\gamma\|_{H^m}
	\end{align*}
	for every $\gamma \in C^\infty(\RZ, \R^n)$.
\end{cor}
\begin{proof}
	By Theorem \ref{thm:QFourier} we can rewrite the Fourier coefficients of $Q$ as 
	\begin{align}\label{formula:Qktilde}
	\widehat{Q\gamma}(k) = \frac{\pi^3}{2} \lambda_k |k|^3 \hat{\gamma} (k) 
	\end{align}
	for all $ k\in \Z$ where $\lambda_k = \frac {\pi} 3 + O(\frac 1k)$ and $\lambda_k >0$ for all $k \not=0.$
	We obtain by (\ref{formula:Qktilde}), the definition of the Sobolev-norm and elementary properties of Fourier coefficients using the positive constant $\widetilde{C}:= \inf_{k\in\Z\setminus{0}} \{\lambda_k^2 2^{-8} \}^{-\frac{1}{2}} < \infty$
	\begin{align*}
	\|\partial^{l}Q\gamma \|_{H^m}^2 &= \sum_{k\in \Z} (1+|k|^2)^m|\widehat{\partial^{l}Q\gamma}(k)|^2 \\
	&= \sum_{k\in \Z} (1+|k|^2)^m (2\pi| k|)^{2l} |\widehat{Q\gamma}(k)|^2 \\
	&= \sum_{k\in \Z} (1+|k|^2)^m (2\pi| k|)^{2l} \frac{\pi^6}{2^2}\lambda_k^2 |k|^6 |\hat{\gamma} (k)|^2\\
	&\geq \inf_{k\in\Z} \{\lambda_k^2 2^{-8} \} \sum_{k \in \Z} (1+|k|^2)^m (2\pi| k|)^{2(l+3)}   |\hat{\gamma} (k)|^2 \\
	&= \widetilde{C}^{-2} \sum_{k \in \Z} (1+|k|^2)^m |\widehat{\partial^{(l+3)}\gamma} (k)|^2 \\
	&= \widetilde{C}^{-2}   \|\partial^{(l+3)}\gamma \|_{H^m}^2,
	\end{align*}
	which proves the statement. 
\end{proof}

\section{Form of $P^T_{\gamma} Q\gamma$ and the bilinear Hilbert transform}\label{sec:bilinearHilberttransform}

Recall that in the first variation of the Möbius energy $\E$ there appears the functional $H$ given by $H\gamma = P^\perp_{\gamma}(\widetilde{H} \gamma)$ where $\widetilde{H}\gamma = Q\gamma+ R_1 \gamma+ R_2 \gamma$ for every closed, simple, and by arc-length parametrized curve $\gamma \in C^\infty(\RZ,\R^n)$. In this section we will see that a type of bilinear Hilbert transform can help to estimate the tangential part of $Q\gamma$, i.e.
\begin{align*}
	(P^T_{\gamma} Q\gamma)(x) :=& P^T_{\gamma(x)} (Q\gamma(x)) = \langle Q\gamma(x), \dot{\gamma}(x) \rangle_{\R^n} \dot{\gamma}(x) \\
	=& \lim_{\varepsilon \downarrow 0}\  \langle Q^\varepsilon\gamma(x), \dot{\gamma}(x) \rangle_{\R^n} \dot{\gamma}(x) = \lim_{\varepsilon \downarrow 0} (P^T_{\gamma} Q^\varepsilon\gamma)(x).
\end{align*} 
In order to avoid problems coming from the singularities of the integrand we will mainly work with the truncated functional $Q^\varepsilon$, $0 < \varepsilon \leq \frac 12$. We use Taylor's approximation up to second order with remainder term in integral from, $\left< \dot{\gamma}(x), \ddot{\gamma}(x) \right> = 0$ for all $x\in \RZ$ coming from the arc-length parametrization of $\gamma$, together with the bilinearity of the scalar product to reach
\begin{align}\label{formula:PTQepsilon}
\begin{split}
\langle Q^\varepsilon\gamma(x), \dot{\gamma}(x) \rangle_{\R^n} & = 2 \int_{|w|\in [\varepsilon, \frac{1}{2}]} \left< 2 \frac{\gamma(x+w)-\gamma(x)-w\dot{\gamma}(x)}{w^2} - \ddot{\gamma}(x) , \dot{\gamma}(x) \right>_{\R^n}  \frac{\mathrm{d}w}{w^2} \\
& = 4  \int_{|w|\in [\varepsilon, \frac{1}{2}]} \int_0^1 (1-t) \left< \frac{\ddot{\gamma}(x+tw) - \ddot{\gamma}(x)}{w^2}, \dot{\gamma}(x) \right>_{\R^n}  \mathrm{d}t \mathrm{d}w \\
&  = 4  \int_{|w|\in [\varepsilon, \frac{1}{2}]} \int_0^1 (1-t)  \frac{\left<\ddot{\gamma}(x+tw), \dot{\gamma}(x) \right>_{\R^n}}{w^2}  \mathrm{d}t \mathrm{d}w \\
& = 4  \int_{|w|\in [\varepsilon, \frac{1}{2}]} \int_0^1 (1-t)  \frac{\left<\ddot{\gamma}(x+tw), \dot{\gamma}(x) - \dot{\gamma}(x+tw)\right>_{\R^n}}{w^2}  \mathrm{d}t \mathrm{d}w \\
& = 4 \int_{|w|\in [\varepsilon, \frac{1}{2}]} \int_0^1  \int_0^1 (1-t) (-t) \frac{\left<\ddot{\gamma}(x+tw), \ddot{\gamma}(x+stw)\right>_{\R^n}}{w}  \mathrm{d}s \mathrm{d}t \mathrm{d}w.
\end{split}
\end{align} 

With $\varepsilon\downarrow 0$ we also get
\begin{align} \label{formula:compTQ} 
(P^T_{\gamma} Q\gamma)(x) := 4 \lim_{\varepsilon \downarrow 0} \int_{|w|\in [\varepsilon, \frac{1}{2}]} \iint_{[0,1 ]^2} (1-t) (-t) \frac{\left<\ddot{\gamma}(x+tw), \ddot{\gamma}(x+stw)\right>_{\R^n}}{w} \dot{\gamma}(x) \mathrm{d}s \mathrm{d}t \mathrm{d}w.
\end{align}
The last terms of (\ref{formula:PTQepsilon}) and  (\ref{formula:compTQ}) motivate us to introduce the following


\begin{defn}
	Let $s_1,s_2 \in [0,1]$. We define the truncated bilinear Hilbert transform on $\RZ$ for an arbitrary $\frac{1}{2} \geq \varepsilon > 0$ by 
	\begin{align*}
	\begin{split}
	& H_{s_1,s_2}^\varepsilon: C^1 (\RZ,\R) \times C^1 (\RZ,\R) \rightarrow L^\infty (\RZ,\R), \\
	& H_{s_1,s_2}^\varepsilon(f,g) (x) := \int_{I\varepsilon} \frac{f(x+s_1w)  g(x+s_2w)}{w} \mathrm{d}w
	\end{split}
	\end{align*} 
	and the bilinear Hilbert transform by 
	\begin{align}\label{defn:Hs1s2}
	\begin{split}
	& H_{s_1,s_2}: C^1 (\RZ,\R) \times C^1 (\RZ,\R) \rightarrow L^\infty (\RZ,\R), \\
	& H_{s_1,s_2}(f,g) (x) := \lim_{\varepsilon \downarrow 0}  H_{s_1,s_2}^{\varepsilon}(f,g)(x).
	\end{split}
	\end{align}
	for all $x\in \RZ$.
\end{defn}

It is straight forward to prove that the bilinear Hilbert transform is well-defined.
 
\begin{lem}\label{lem:Hs1s2cont}
	For every $s_1,s_2 \in [0,1]$ the bilinear Hilbert transform $H_{s_1,s_2}$ as well as the truncated bilinear Hilbert transform $H_{s_1,s_2}^\varepsilon$, $0 < \varepsilon \leq \frac 12$, are indeed bilinear, well-defined and continuous on $C^1 (\RZ,\R)\times C^1 (\RZ,\R)$  to $L^\infty (\RZ,\R)$.  
\end{lem}
\begin{proof}
	Let $s_1,s_2 \in [0,1]$. It is easy to see due to the linearity of the integral in (\ref{defn:Hs1s2}) that $H_{s_1,s_2}$ is indeed linear in both components. $H_{s_1,s_2}$ is also well-defined for every $f,g \in C^1 (\RZ,\R)$, which we can see by inserting two times a zero and by Lipschitz-continuity of continuously differentiable functions on a compact set such that
	\begin{align}\label{lem:Hs1s2}
	\begin{split}
	& \left|H_{s_1,s_2}(f,g)(x)\right| \\
	& = \left| \lim_{\varepsilon \downarrow 0} \int_{|w|\in [\varepsilon, \frac{1}{2}]} \frac{f(x+s_1w)  g(x+s_2w)}{w} - \frac{f(x)  g(x)}{w}  \mathrm{d}w \right| \\
	& \leq \lim_{\varepsilon \downarrow 0} \int_{|w|\in [\varepsilon, \frac{1}{2}]} \frac{|f(x+s_1w)  g(x+s_2w) - f(x)g(x+s_2w) +f(x)g(x+s_2w) -f(x)  g(x)|}{|w|}  \mathrm{d}w \\
	& \leq \lim_{\varepsilon \downarrow 0} \int_{|w|\in [\varepsilon, \frac{1}{2}]} \frac{|f(x+s_1w) - f(x)| | g(x+s_2w)| +  |f(x)| |g(x+s_2w)- g(x)|}{|w|}  \mathrm{d}w \\
	& \leq \lim_{\varepsilon \downarrow 0} \int_{|w|\in [\varepsilon, \frac{1}{2}]} \frac{\|f'\|_\infty s_1|w| \|g\|_\infty +  \|f\|_\infty \|g'\|_\infty s_2 |w|}{|w|}  \mathrm{d}w \\
	& \leq C (s_1,s_2)  \|f\|_{C^1} \|g\|_{C^1}  < \infty
	\end{split}
	\end{align}
	for all $x\in [0,1]$ and $0 < C(s_1,s_2) < \infty$. Hence we can deduce from (\ref{lem:Hs1s2}) that $H_{s_1,s_2}$ is continuous. 
	By the same arguments we get the same properties for $H_{s_1,s_2}^\epsilon$ for any $\frac{1}{2} \geq \epsilon > 0 $.
\end{proof}

The next theorem allows us to estimate the truncated bilinear Hilbert transform in a way that does not depend on the truncation parameter $\varepsilon$. It will later be used to estimate the last term of (\ref{formula:PTQepsilon}).

\begin{thm}\label{thm:BilinearHT}
	Let $m> \frac{1}{2}$ and $\frac{1}{2} \geq \varepsilon > 0$. Then there exists a positive constant $C_H < \infty$ independent of $\varepsilon$ such that 
	\begin{align}\label{formula:BilinearHT}
	\|H_{s_1,s_2}^\varepsilon(f,g)\|_{H^m} \leq C_H \|f\|_{H^m} \|g\|_{H^m}
	\end{align}
	for all $s_1,s_2 \in [0,1]$ and $f,g \in  C^\infty(\RZ,\R^n)$.
\end{thm}

In order to prove Theorem \ref{thm:BilinearHT}, we require Young's inequality for the convolution of series and Sobolev's inequality which we state for the reader's convenience. We consider for $p \in [1,\infty)$ the sequence space 
\begin{align*}
\ell^p := \{x = (x_k)_{k\in\Z} \in \C^\Z \ | \ \|x\|_{\ell^p}:= \left(\sum_{k\in\Z} |x_k|^p\right)^{\frac{1}{p}} < \infty \},
\end{align*}
which is a Banach space. The following is well known (cf. \cite[Theorem 20.18]{HR79}).

\begin{lem}[Young's inequality]\label{lem:Young}
	Let $p,q,r \geq 1$ such that $\frac{1}{p} + \frac{1}{q}=\frac{1}{r}+1$. For $x\in \ell^p$ and $y\in  \ell^q$ we  let $(x*y)(k):= \sum_{n\in \Z} x(n) y(k-n)$ denote the convolution. Then $x*y \in \ell^r$ and
	\begin{align}\label{lem:x*y}
	\|x*y\|_{\ell^r} \leq \|x\|_{\ell^p} \|y\|_{\ell^q} 
	\end{align}
\end{lem}

We will use Young's inequality to prove the following version of the Sobolev inequality:

\begin{lem}[Sobolev's lemma]\label{lem:Sobolev}
	Let $f\in H^m(\RZ, \R^n)$ for $m> \frac{1}{2}$. Then $\hat{f}:= (\hat{f}(k))_{k\in\Z} \in \ell^1$ and there exists a positive constant $C_0 < \infty$ such that
	\begin{align*}
	\|\hat{f}\|_{\ell^1} \leq C_0 \|f\|_{H^m}.
	\end{align*}
\end{lem}
\begin{proof}
	Simply by the definition of the $\ell^1$-norm and by Hölder's inequality we yield
	\begin{align*}
	\|\hat{f}\|_{\ell^1} & =   \sum_{k\in\Z} |\hat{f}(k)|  \\
	& = \sum_{k\in\Z} |\hat{f}(k)| (1+k^2)^{\frac{m}{2}} (1+k^2)^{-\frac{m}{2}} \\
	& \leq \left( \sum_{k\in\Z} |\hat{f}(k)|^2 (1+k^2)^{m}  \right)^{\frac{1}{2}} \left( \sum_{k\in\Z} \frac{1}{(1+k^2)^{m}} \right)^{\frac{1}{2}} \\
	& = C_0 \|f\|_{H^m}
	\end{align*}
	where $C_0:=\left( \sum_{k\in\Z} \frac{1}{(1+k^2)^{m}} \right)^{\frac{1}{2}}$ which is finite since $m> \frac{1}{2}$.
\end{proof}

Now we are ready to prove Theorem  \ref{thm:BilinearHT}.

\begin{proof}[Proof of Theorem \ref{thm:BilinearHT}] 
	Let $s_1,s_2 \in [0,1]$, $m > \frac{1}{2}$ and $\frac{1}{2} \geq \varepsilon > 0$. Moreover, let $f,g \in  C^\infty(\RZ,\R)\subseteq H^m(\RZ,\R)$. Due to $m > \frac{1}{2}$ and \cite[Theorem 8.2]{DNPV12}, the partial sums of the Fourier series converge uniformly to the functions $f$ and $g$, more precisely
	\begin{align}\label{formula:apprfg}
	\begin{split}
	\|f-p_n \|_\infty & \overset{n\rightarrow \infty}{\longrightarrow} 0 \textnormal{ and } \\
	\|g-q_n \|_\infty & \overset{n\rightarrow \infty}{\longrightarrow} 0
	\end{split}		
	\end{align}
	where $	p_n (x) := \sum_{k=-n}^{n} \hat{f}(k) e^{-2\pi i k x}$, $q_n (x):=\sum_{k=-n}^{n} \hat{g}(k) e^{-2\pi i k x} $ for all $x\in \RZ$ and $n\in\N_0$, and $p_n, q_n \in C^\infty (\RZ, \R)$.
	
	The strategy of the proof is to first compute and estimate the Fourier coefficient of the truncated bilinear Hilbert transform of $p_n, q_n$. Then we derive the main estimate (\ref{formula:BilinearHT}) for $p_n,q_n$, $n\in \N$, and pass to the limit $n\rightarrow \infty$ to obtain the main estimate (\ref{formula:BilinearHT}) for $f,g$.
	Let $n\in \N$. So we start by calculating the $k$-th Fourier coefficient of $H_{s_1,s_2}^\varepsilon(p_n,q_n)$, $k\in\Z$, more precisely
	\begin{align}\label{formula:HepsilonFourier}
	\widehat{H_{s_1,s_2}^\varepsilon(p_n,q_n)}(k) = \int_0^1 \int_{|w|\in [\varepsilon, \frac{1}{2}]} \frac{p_n(x+s_1w)  q_n(x+s_2w)}{w} e^{-2\pi i k x} \mathrm{d}w \mathrm{d}x.
	\end{align}
	Below we will interchange the integrals of (\ref{formula:HepsilonFourier}) two times, which is allowed since the integrand is smooth on $[-\frac 12, -\varepsilon] \cup [\varepsilon, \frac{1}{2}]$. By using that the Fourier coefficients of a product of two functions is the convolution of the Fourier coefficients of the two functions we get
	\begin{align}\label{formula:pnqnepsilonI}
	\begin{split}
	& \int_{|w|\in [\varepsilon, \frac{1}{2}]} \int_0^1  p_n(x+s_1w)  q_n(x+s_2w) e^{-2\pi i k x}  \mathrm{d}x \frac{\mathrm{d}w }{w} \\
	& = \int_{|w|\in [\varepsilon, \frac{1}{2}]}  \sum_{l\in\N} \widehat{p_n}(l) e^{2\pi ils_1 w} \widehat{q_n}(k-l) e^{2\pi i(k-l)s_2 w} \frac{\mathrm{d}w }{w} \\ 
	& = 
	\begin{cases}
	0       & \quad \text{if } |k| > 2n, \\
	\int_{|w|\in [\varepsilon, \frac{1}{2}]}  \sum_{l=-n}^{n} e^{2\pi i(ls_1+ (k-l)s_2) w} \widehat{f}(l) \widehat{g}(k-l) \frac{\mathrm{d}w }{w}  & \quad \text{if } |k| \leq 2n.
	\end{cases} 
	\end{split}
	\end{align}
	By a simple substitution we gain for $|k|\leq 2n$
	\begin{align}\label{formula:pnqnepsilonII}
	\begin{split}
	& \int_{|w|\in [\varepsilon, \frac{1}{2}]}  \sum_{l=-n}^{n} e^{2\pi i(ls_1+ (k-l)s_2) w} \widehat{f}(l) \widehat{g}(k-l) \frac{\mathrm{d}w }{w} \\ 
	& = \int_\varepsilon^{\frac{1}{2}}  \sum_{l=-n}^{n} \frac{e^{2\pi i(ls_1+ (k-l)s_2) w} - e^{2\pi i(ls_1+ (k-l)s_2) w}}{w} \widehat{f}(l) \widehat{g}(k-l) \mathrm{d}w \\
	& = \int_\varepsilon^{\frac{1}{2}}  \sum_{l=-n}^{n} \frac{e^{i w\phi_{l,k}} - e^{-i w \phi_{l,k}}}{w} \widehat{f}(l) \widehat{g}(k-l) \mathrm{d}w \\
	& = \int_\varepsilon^{\frac{1}{2}}  \sum_{l=-n}^{n}\frac{2i \sin (w \phi_{l,k})}{w} \widehat{f}(l) \widehat{g}(k-l) \mathrm{d}w
	\end{split}
	\end{align}
	where $\phi_{n,k} := 2\pi (ls_1 +(k-l)s_2) \in \R$. 
	By Interchanging integral and sum we get
	\begin{align}\label{formula:Sineintegral}
	\begin{split}
	&  \sum_{l=-n}^{n} \widehat{f}(l) \widehat{g}(k-l) 2i \int_\varepsilon^{\frac{1}{2}} \frac{ \sin (w \phi_{l,k})}{w}  \mathrm{d}w \\
	& = \sum_{l=-n}^{n} \widehat{f}(l) \widehat{g}(k-l) 2i \int_{\phi_{l,k}\varepsilon}^{\frac{\phi_{l,k}}{2}} \frac{ \sin (\widetilde{w})}{\widetilde{w}}  \mathrm{d}\widetilde{w} \\
	& =  \sum_{l=-n}^{n} \widehat{f}(l) \widehat{g}(k-l) 2i \left(\Si\left(\frac{\phi_{l,k}}{2}\right) - \Si(\phi_{l,k}\varepsilon) \right),
	\end{split}
	\end{align}
	where the sine integral is defined by $\Si (x):= \int_0^x \frac{\sin t}{t}\mathrm{d}t$. 
	Now we estimate by using (\ref{formula:pnqnepsilonI}), (\ref{formula:pnqnepsilonII}), and (\ref{formula:Sineintegral}) the k-th Fourier coefficient of the approximation of the bilinear Hilbert transform $H_{s_1,s_2}^\varepsilon$ as follows
	\begin{align}\label{formula:Hfouriercoeffpnqn}
	\begin{split}
	\left|\widehat{H_{s_1,s_2}^\varepsilon(p_n,q_n)}(k) \right| & = \left| \int_0^1 \int_{|w|\in [\varepsilon, \frac{1}{2}]} \frac{ p_n(x+s_1w)  q_n(x+s_2w) e^{-2\pi i k x}}{w}  \mathrm{d}w \mathrm{d}x \right| \\
	& =  \left|  \int_{|w|\in [\varepsilon, \frac{1}{2}]} \int_0^1  \frac{ p_n(x+s_1w)  q_n(x+s_2w) e^{-2\pi i k x}}{w}   \mathrm{d}x \mathrm{d}w \right | \\
	& \leq \int_{|w|\in [\varepsilon, \frac{1}{2}]} \int_0^1  \left| \frac{ p_n(x+s_1w)  q_n(x+s_2w) e^{-2\pi i k x}}{w} \right|   \mathrm{d}x\mathrm{d}w\\
	& \leq \sum_{l=-n}^{n} |\widehat{f}(l)| | \widehat{g}(k-l)| 2 \left|\Si\left(\frac{\phi_{l,k}}{2}\right) - \Si(\phi_{l,k}\varepsilon) \right| \\
	& \leq \sum_{l=-n}^{n} |\widehat{f}(l)| | \widehat{g}(k-l)| 2 \left(\left|\Si\left(\frac{\phi_{l,k}}{2}\right) \right| + \left| \Si(\phi_{l,k}\varepsilon) \right| \right) \\
	& \leq M \sum_{l=-n}^{n} |\widehat{f}(l)| | \widehat{g}(k-l)| \\
	& \leq M  (|p_n| * |q_n|) (k)
	\end{split}
	\end{align}
	where $M := 4 \sup_{\phi \in [0,\infty [} \Si (\phi) <\infty$. 
	
	Moreover, we can deduce from (\ref{formula:Hfouriercoeffpnqn}) using the elementary estimate
	$$
	(k^2+1)^{\frac{m}{2}} \leq 2^m \left((l^2+1)^{\frac{m}{2}} + ((k-l)^2+1)^{\frac{m}{2}} \right),
	$$
	which immediately follows from $|k| \leq 2 \min\{|l|, |k-l|\}$, that
	\begin{align}\label{formula:Hs1s2decompose}
	\begin{split}
	\left|(k^2 + 1)^{\frac{m}{2}}\widehat{H_{s_1,s_2}(p_n,q_n)}(k) \right| \leq & M \left( (k^2 + 1)^{\frac{m}{2}} (|p_n| * |q_n|) (k)\right) \\
	\leq & MC_m \sum_{l\in\N} (l^2 + 1)^{\frac{m}{2}} \left|\widehat{p_n}(l)\right| \left| \widehat{q_n}(k-l) \right|  \\
	& + MC_m \sum_{l\in\N}  \left|\widehat{p_n}(l)\right| ((k-l)^2 + 1)^{\frac{m}{2}} \left| \widehat{q_n}(k-l) \right|.
	\end{split}
	\end{align}
	This motivates us to introduce the following two series component-wise for all $k\in\Z$
	\begin{align*}
	\widetilde{p_n}(k) &:= (k^2 + 1)^{\frac{m}{2}} \left|\widehat{p_n}(k)\right| \textnormal{ and } \\
	\widetilde{q_n}(k) &:= (k^2 + 1)^{\frac{m}{2}} \left|\widehat{q_n}(k)\right|
	\end{align*}
	such that  (\ref{formula:Hs1s2decompose}) can be written as
	\begin{align*}
	\left|(k^2 + 1)^{\frac{m}{2}}\widehat{H_{s_1,s_2}^\varepsilon(p_n,q_n)}(k) \right| \leq MC_m \left( (\widetilde{p_n}*\left| \widehat{q_n}\right|) (k) +  (|\widehat{p_n}|*\widetilde{q_n} ) (k)\right).
	\end{align*}
	In the end we get the desired estimate (\ref{formula:BilinearHT}) for $p_n, q_n$ by applying the definition of the Sobolev norm, triangle inequality, Young's inequality with $1+\frac{1}{2}= \frac{1}{2} +1$ in Lemma \ref{lem:Young} and Sobolev's lemma \ref{lem:Sobolev} such that
	\begin{align}\label{formula:mainestimatepnqn}
	\begin{split}
	\|H_{s_1,s_2}^\varepsilon(p_n,q_n)\|_{H^m} & = \|( (k^2 + 1)^{\frac{m}{2}}\widehat{H_{s_1,s_2}^\varepsilon(p_n,q_n)}(k))_{k\in\Z} \|_{\ell^2} \\
	& \leq MC_m \left( \| ((\widetilde{p_n}*\left| \widehat{q_n}\right|) (k))_{k\in\Z} \|_{\ell^2}  + \|((|\widehat{p_n}|*\widetilde{q_n} ) (k))_{k\in\Z} \|_{\ell^2} \right) \\
	& \leq MC_m \left(\|\widetilde{p_n}\|_{\ell^2} \|\widehat{q_n}\|_{\ell^1} + \|\widetilde{q_n}\|_{\ell^2} \|\widehat{p_n}\|_{\ell^1} \right) \\
	& \leq MC_mC_0 \left(\|p_n\|_{H^m} \|q_n\|_{H^m} + \|p_n\|_{H^m} \|q_n\|_{H^m}\right) \\
	& = C_H \|p_n\|_{H^m} \|q_n\|_{H^m} 
	\end{split}
	\end{align}
	where $C_H:= 2MC_mC_0$ is a constant only depending on $m$, $s_1$ and $s_2$. This proves (\ref{formula:BilinearHT}) for $p_n$ and $q_n$ instead of $f$ and $g$. 
	
	Finally, we come to the last step, passing to the limit $n\rightarrow \infty$. First of all we gain for any $n\in \N$
	\begin{align}\label{formula:Hfg}
	\begin{split}
	\|H_{s_1,s_2}^\varepsilon(p_{n},q_{n})\|_{H^m} & \leq C_H \|p_{n}\|_{H^m} \|q_{n}\|_{H^m} \\
	&  \leq C_H \|f\|_{H^m} \|g\|_{H^m} 
	\end{split}
	\end{align}
	by (\ref{formula:mainestimatepnqn}) and the definition of $p_n$ and $q_n$.
	In order to pass to the limit on the left-hand side of (\ref{formula:Hfg}), we check the uniform convergence
	\begin{align*}
	0 & \leq  \|H_{s_1,s_2}^\varepsilon(f,g) - H_{s_1,s_2}^\varepsilon(p_n,q_n)\|_\infty \\
	& \leq \|H_{s_1,s_2}^\varepsilon(f,g) - H_{s_1,s_2}^\varepsilon(p_n,g)\|_\infty + \|H_{s_1,s_2}^\varepsilon(p_n,g) - H_{s_1,s_2}^\varepsilon(p_n,q_n)\|_\infty \\
	& = \|H_{s_1,s_2}^\varepsilon(f-p_n,g)\|_\infty + \|H_{s_1,s_2}^\varepsilon(p_n,g-q_n)\|_\infty \\
	& \leq C(s_1,s_2) \underbrace{\|f-p_n\|_\infty}_{\overset{n\rightarrow \infty}{\longrightarrow}0} \|g\|_\infty + C(s_1,s_2) \|p_n\|_\infty \underbrace{\|g-q_n\|_\infty}_{\overset{n\rightarrow \infty}{\longrightarrow}0}\\
	& \overset{n\rightarrow \infty}{\longrightarrow} 0
	\end{align*}
	by Lemma \ref{lem:Hs1s2cont} and (\ref{formula:apprfg}) so that we may infer by the uniqueness of the Fourier coefficients of $ H_{s_1,s_2}^\varepsilon(p_n,q_n) \in H^m (\RZ,\R) \subseteq L^2(\RZ,\R)$ that for all $k\in\Z$
	\begin{align*}
	\left|\widehat{H_{s_1,s_2}^\varepsilon(f,g)}(k)- \widehat{H_{s_1,s_2}^\varepsilon(p_n,q_n)}(k)\right| \overset{n\rightarrow \infty}{\longrightarrow} 0.
	\end{align*}
	Thus, we get for all $N\in \N$
	\begin{align*}
	\left| \left( \sum_{|k|\leq N} (1+|k|^2)^m|\widehat{H_{s_1,s_2}^\varepsilon(p_n,q_n)}(k)|^2 \right)^{\frac{1}{2}} - \left( \sum_{|k|\leq N} (1+|k|^2)^m|\widehat{H_{s_1,s_2}^\varepsilon(f,g)}(k)|^2 \right)^{\frac{1}{2}} \right| \overset{n\rightarrow \infty}{\longrightarrow} 0 ,
	\end{align*}
	which implies
	\begin{align*}
	\left( \sum_{|k|\leq N} (1+|k|^2)^m|\widehat{H_{s_1,s_2}^\varepsilon(f,g)}(k)|^2 \right)^{\frac{1}{2}} \leq C_H \|f\|_{H^m} \|g\|_{H^m} 
	\end{align*}
	by (\ref{formula:Hfg}).	Therefore, it also holds 
	\begin{align*}
	\|H_{s_1,s_2}^\varepsilon(f,g)\|_{H^m}\leq C_H \|f\|_{H^m} \|g\|_{H^m} 
	\end{align*}
	and the main estimate (\ref{formula:BilinearHT}) is proven.
\end{proof}

\section{Form of the remaining parts}\label{sec:remaining parts}
 
We aim to transform the truncated remaining parts $R_1^\varepsilon$ and $R_2^\varepsilon$ of the decomposition of the Möbius energy $\E$ in Section \ref{section:decomposition} into multiple integrals of analytic functions of several variables. This helps us to estimate the derivatives of the orthogonal projections of $R_1^\varepsilon$ and $R_2^\varepsilon$ later.

Let $\frac{1}{2} \geq\varepsilon > 0$ and $x\in [0,1]$. We start by recalling the definition of the first remaining part
\begin{align*}
(R_1 \gamma)(x)= \lim_{\varepsilon\downarrow 0} (R_1^\varepsilon \gamma)(x)
\end{align*}
where 
\begin{align*}
(R_1^\varepsilon \gamma)(x):= 4  \int_{|w|\in [\varepsilon, \frac{1}{2}]} \left( \frac{1}{|\gamma(x+w)-\gamma(x)|^4} - \frac{1}{w^4} \right) (\gamma(x+w)-\gamma(x)-w\dot{\gamma}(x)) \mathrm{d} w.
\end{align*}
Similarly,
\begin{align*}
(R_2 \gamma)(x)= \lim_{\varepsilon\downarrow 0} (R_2^\varepsilon \gamma)(x), 
\end{align*}
where
\begin{align*}
(R_2^\varepsilon \gamma)(x):= -2 \int_{|w|\in [\varepsilon, \frac{1}{2}]} \left( \frac{1}{|\gamma(x+w)-\gamma(x)|^2} - \frac{1}{w^2} \right) \ddot{\gamma}(x) \mathrm{d} w.
\end{align*}
Because the second remaining part is the more elementary one, we start by working with $P^\perp_{\dot{\gamma}} R_2^\varepsilon \gamma $ now.

\subsection{Form of $P^\perp_{\dot{\gamma}} R_2^\varepsilon \gamma $}

\begin{thm}
	We can express $P^\perp_{\dot{\gamma}} R_2^\varepsilon \gamma$ as multiple integral
	\begin{align*}
	(P^\perp_{\dot{\gamma}} R_2^\varepsilon \gamma )(x) =  \int_{|w|\in [\varepsilon, \frac{1}{2}]}  \iiiint_{[0,1]^4} (s_1-s_2)^2 G_2 (\cdots) (x) \mathrm{d} \phi_1 \mathrm{d} \phi_2 \mathrm{d} s_1 \mathrm{d} s_2 \mathrm{d} w,
	\end{align*}
	over the analytic function $G_2 : \R^n \setminus \{0\} \times \R^{3n} \rightarrow \R^n$, 
	\begin{align*}
	G_2 (\cdots) := P^\perp_{\dot{\gamma}} \widetilde{G_2} (\cdots) = \widetilde{G_2} (\cdots) - \langle \widetilde{G_2} (\cdots), \dot{\gamma}\rangle \dot{\gamma},
	\end{align*}
	where $\widetilde{G_2}: \R^n \setminus \{0\} \times \R^{3n}  \rightarrow \R^n$ defined as $\widetilde{G_2}(a, x, y, z):= -  \frac{1}{|a|^2} \left<x,y \right>_{\R^n} z$ is analytic and $\widetilde{G}_2(\cdots)$ is an abbreviation for 
	\begin{align*}
	\widetilde{G_2}\left( \int_0^1 \dot{\gamma} (\cdot+tw) \mathrm{d} t, \ddot{\gamma}(\cdot+s_2w + (s_1-s_2)\phi_1 w), \ddot{\gamma}(\cdot+s_2w + (s_1-s_2)\phi_2 w), \ddot{\gamma}(\cdot)  \right) .
	\end{align*}
\end{thm}

\begin{proof}

We begin to compute the left-hand side part of the integrand of $R_2^\varepsilon \gamma$ by the fundamental theorem of calculus, the arc-length parametrization of $\gamma$, and linearity of the scalar product as 
\begin{align}\label{formula:R2gammaint1}
\begin{split}
\frac{1}{|\gamma(x+w)-\gamma(x)|^2} - & \frac{1}{w^2}  = \frac{w^2}{|\gamma(x+w)-\gamma(x)|^2} \frac{1- \frac{|\gamma(x+w)-\gamma(x)|^2}{w^2}}{w^2} \\
& = g_2\left(\int_0^1 \dot{\gamma} (x+tw) \mathrm{d} t \right) \frac{2\left(1- \frac{|\gamma(x+w)-\gamma(x)|^2}{w^2}\right)}{w^2} \\
& = g_2\left(\int_0^1 \dot{\gamma} (x+tw) \mathrm{d} t \right) \frac{\iint_{[0,1]^2} |\dot{\gamma}(x+s_1w)-\dot{\gamma}(x+s_2w)|^2 \mathrm{d} s_1 \mathrm{d} s_2}{w^2}
\end{split}
\end{align}
where $g_2(x) = \frac{1}{2|x|^2}$ for all $x\in \R^n \setminus\{0\}$ is analytic away from the origin and furthermore,
\begin{align}\label{formula:R2gammaint2}
\begin{split}
& \frac{|\dot{\gamma}(x+s_1w)-\dot{\gamma}(x+s_2w)|^2}{w^2} = \\ 
& (s_1 - s_2)^2 \int_0^1 \int_0^1 \left< \ddot{\gamma}(x+s_2w + (s_1-s_2)\phi_1 w), \ddot{\gamma}(x+s_2w + (s_1-s_2)\phi_2 w) \right>_{\R^n} \mathrm{d} \phi_1 \mathrm{d} \phi_2.
\end{split} 
\end{align}
Therefore, we can rewrite $R_2^\varepsilon \gamma$ such that 
\begin{align*}
(R_2^\varepsilon & \gamma) (x) =  \int_{|w|\in [\varepsilon, \frac{1}{2}]}  \iiiint_{[0,1]^4} (s_1-s_2)^2 \cdot	\\
& \widetilde{G_2}\left( \int_0^1 \dot{\gamma} (x+tw) \mathrm{d} t, \ddot{\gamma}(x+s_2w + (s_1-s_2)\phi_1 w), \ddot{\gamma}(x+s_2w + (s_1-s_2)\phi_2 w), \ddot{\gamma}(x)  \right) \\
& \mathrm{d} \phi_1 \mathrm{d} \phi_2 \mathrm{d} s_1 \mathrm{d} s_2 \mathrm{d} w
\end{align*}
where $\widetilde{G_2}: \R^n \setminus \{0\} \times \R^{3n}  \rightarrow \R^n$ is defined as 
\begin{align*}
\widetilde{G_2}(a, x, y, z):= -2 g_2(a) \left<x,y \right>_{\R^n} z
\end{align*}
and is again an analytic function away from the origin in the first variable, because sums, products and compositions of analytic functions are analytic as well. 
Now by applying the orthogonal projection $P^\perp_{\dot{\gamma}} $ on $ R_2^\varepsilon \gamma$ we obtain
\begin{align*}
(P^\perp_{\dot{\gamma}} R_2^\varepsilon \gamma )(x) =  \int_{|w|\in [\varepsilon, \frac{1}{2}]}  \iiiint_{[0,1]^4} (s_1-s_2)^2 (P^\perp_{\dot{\gamma}} \widetilde{G_2} (\cdots) ) (x) \mathrm{d} \phi_1 \mathrm{d} \phi_2 \mathrm{d} s_1 \mathrm{d} s_2 \mathrm{d} w
\end{align*}
where $\widetilde{G_2} (\cdots)$ is an abbreviation for 
\begin{align*}
\widetilde{G_2}\left( \int_0^1 \dot{\gamma} (\cdot+tw) \mathrm{d} t, \ddot{\gamma}(\cdot+s_2w + (s_1-s_2)\phi_1 w), \ddot{\gamma}(\cdot+s_2w + (s_1-s_2)\phi_2 w), \ddot{\gamma}(\cdot)  \right) .
\end{align*}
Finally we can express $P^\perp_{\dot{\gamma}} R_2^\varepsilon \gamma$ as multiple integral
\begin{align*}
(P^\perp_{\dot{\gamma}} R_2^\varepsilon \gamma )(x) =  \int_{|w|\in [\varepsilon, \frac{1}{2}]}  \iiiint_{[0,1]^4} (s_1-s_2)^2 G_2 (\cdots) (x) \mathrm{d} \phi_1 \mathrm{d} \phi_2 \mathrm{d} s_1 \mathrm{d} s_2 \mathrm{d} w,
\end{align*}
over $G_2 : \R^n \setminus \{0\} \times \R^{3n} \rightarrow \R^n$, 
\begin{align*}
G_2 (\cdots) := P^\perp_{\dot{\gamma}} \widetilde{G_2} (\cdots) = \widetilde{G_2} (\cdots) - \langle \widetilde{G_2} (\cdots), \dot{\gamma}\rangle \dot{\gamma},
\end{align*}
which is clearly analytic again. 
\end{proof}


\subsection{Form of $P^\perp_{\dot{\gamma}} R_1^\varepsilon \gamma $}

Now we may transfer the previous calculations to the orthogonal projection $P^\perp_{\dot{\gamma}} R_1^\varepsilon \gamma $ of the first remaining part.

\begin{thm}
	We can express $P^\perp_{\dot{\gamma}} R_1^\varepsilon \gamma$ as multiple integral
	\begin{align*}
	(P^\perp_{\dot{\gamma}} R_1^\varepsilon \gamma )(x) =   \int_{|w|\in [\varepsilon, \frac{1}{2}]}  \iiiint_{[0,1]^4}  (r_1-r_2)^2  G_1 (\cdots) (x) \mathrm{d} \psi_1 \mathrm{d} \psi_2 \mathrm{d} r_1 \mathrm{d} r_2 \mathrm{d} w,
	\end{align*}
	over the analytic function $G_1 : \R^n \setminus \{0\} \times \R^{3n} \rightarrow \R^n$, 
	\begin{align}\label{formula:G1}
	G_1 (\cdots) := P^\perp_{\dot{\gamma}} \widetilde{G_1} (\cdots) = \widetilde{G_1} (\cdots) - \langle \widetilde{G_1} (\cdots), \dot{\gamma}\rangle \dot{\gamma},
	\end{align}
	where $\widetilde{G_1} : \R^n\setminus \{0\} \times \R^{3n} \rightarrow \R^n$ defined as $\widetilde{G_1} (a,x,y,z) := 42 (\frac{1}{ |a|^4} + \frac{1}{ |a|^2}) \left<x,y \right>_{\R^n} z $ is analytic and $\widetilde{G_1} (\cdots)$ is an abbreviation for 
	\begin{align}\label{formula:tildeG1}
	\begin{split}
	\widetilde{G_1} \bigg( & \int_0^1 \dot{\gamma}(\cdot+sw)\mathrm{d}s, \ddot{\gamma}(\cdot+r_2w + (r_1-r_2)\psi_1 w), \\
	& \ddot{\gamma}(\cdot+r_2w + (r_1-r_2)\psi_2 w), \int_0^1 \ddot{\gamma}(\cdot+tw)(1-t) \mathrm{d} t \bigg) .
	\end{split}
	\end{align}
\end{thm}

\begin{proof}
Similarly to the previous subsection we start by computing the integrand of $R_1^\varepsilon \gamma$ by applying the integral form of the remainder of a Taylor polynomial 
\begin{align}\label{formula:R1integrand}
\begin{split}
& \left( \frac{1}{|\gamma(x+w)-\gamma(x)|^4}  - \frac{1}{w^4} \right) (\gamma(x+w)-\gamma(x)-w\dot{\gamma}(x)) \\
& = \left( \frac{1}{|\gamma(x+w)-\gamma(x)|^4} - \frac{1}{w^4} \right) w^2 \int_0^1 \ddot{\gamma}(x+tw)(1-t) \mathrm{d} t \\
& = \left( \frac{w^2}{|\gamma(x+w)-\gamma(x)|^4} - \frac{1}{w^2} \right) \int_0^1 \ddot{\gamma}(x+tw)(1-t) \mathrm{d} t.
\end{split}
\end{align}
By considering the fundamental theorem of calculus we may transform the term within the brackets in (\ref{formula:R1integrand}) as
\begin{align*}
& \frac{w^2}{|\gamma(x+w)-\gamma(x)|^4} - \frac{1}{w^2}  = \frac{w^4}{|\gamma(x+w)-\gamma(x)|^4} \left(\frac{1-\frac{|\gamma(x+w)-\gamma(x)|^4}{w^4} }{w^2}\right) \\
& = \frac{w^4}{|\gamma(x+w)-\gamma(x)|^4} \frac{1}{w^2} \left(1+\frac{|\gamma(x+w)-\gamma(x)|^2}{w^2} \right) \left(1-\frac{|\gamma(x+w)-\gamma(x)|^2}{w^2} \right) \\
& = \left(\frac{w^4}{|\gamma(x+w)-\gamma(x)|^4} + \frac{w^2}{|\gamma(x+w)-\gamma(x)|^2}\right) \frac{1}{w^2} \left(1-\frac{|\gamma(x+w)-\gamma(x)|^2}{w^2} \right) \\
& = g_1\left(\int_0^1 \dot{\gamma}(x+sw)\mathrm{d}s\right)  \underbrace{\frac{2}{w^2} \left(1-\frac{|\gamma(x+w)-\gamma(x)|^2}{w^2} \right)}_{:= E}
\end{align*}
where $g_1(x) = \frac{1}{2 |x|^4} + \frac{1}{2 |x|^2}$ for all $x\in \R^n\setminus \{0\}$ is analytic away from the origin. Furthermore, by using (\ref{formula:R2gammaint1}) and (\ref{formula:R2gammaint2}) for computing $E$, we arrive at
\begin{align*}
(R_1^\varepsilon \gamma) (x) = &  \int_{|w|\in [\varepsilon, \frac{1}{2}]}  \iiiint_{[0,1]^4} (r_1-r_2)^2 \cdot	\\
& \widetilde{G_1} \bigg( \int_0^1 \dot{\gamma}(x+sw)\mathrm{d}s, \ddot{\gamma}(x+r_2w + (r_1-r_2)\psi_1 w), \ddot{\gamma}(x+r_2w + (r_1-r_2)\psi_2 w), \\
& \int_0^1 \ddot{\gamma}(x+tw)(1-t) \mathrm{d} t \bigg) \mathrm{d} \psi_1 \mathrm{d} \psi_2 \mathrm{d} r_1 \mathrm{d} r_2 \mathrm{d} w
\end{align*}
where $\widetilde{G_1} : \R^n\setminus \{0\} \times \R^{3n} \rightarrow \R^n$ such that 
\begin{align*}
\widetilde{G_1} (a,x,y,z) = 4 g_1(a) \left<x,y \right>_{\R^n} z 
\end{align*}
and therefore analytic away from the origin in the first variable.
By applying the orthogonal projection $P^\perp_{\dot{\gamma}} $ on $ R_1^\varepsilon \gamma$ we get
\begin{align*}
(P^\perp_{\dot{\gamma}} R_1^\varepsilon \gamma )(x) =  \int_{|w|\in [\varepsilon, \frac{1}{2}]}  \iiiint_{[0,1]^4}  (r_1-r_2)^2  (P^\perp_{\dot{\gamma}} \widetilde{G_1} (\cdots) ) (x) \mathrm{d} \psi_1 \mathrm{d} \psi_2 \mathrm{d} r_1 \mathrm{d} r_2 \mathrm{d} w
\end{align*}
where $\widetilde{G_1} (\cdots)$ is an abbreviation for 
\begin{align*}
\widetilde{G_1} \bigg( & \int_0^1 \dot{\gamma}(\cdot+sw)\mathrm{d}s, \ddot{\gamma}(\cdot+r_2w + (r_1-r_2)\psi_1 w), \\
& \ddot{\gamma}(\cdot+r_2w + (r_1-r_2)\psi_2 w), \int_0^1 \ddot{\gamma}(\cdot+tw)(1-t) \mathrm{d} t \bigg) .
\end{align*}
In the end, we can also express $P^\perp_{\dot{\gamma}} R_1^\varepsilon \gamma$ as multiple integral
\begin{align*}
(P^\perp_{\dot{\gamma}} R_1^\varepsilon \gamma )(x) =   \int_{|w|\in [\varepsilon, \frac{1}{2}]}  \iiiint_{[0,1]^4}  (r_1-r_2)^2  G_1 (\cdots) (x) \mathrm{d} \psi_1 \mathrm{d} \psi_2 \mathrm{d} r_1 \mathrm{d} r_2 \mathrm{d} w,
\end{align*}
over $G_1 : \R^n \setminus \{0\} \times \R^{3n} \rightarrow \R^n$, 
\begin{align*}
G_1 (\cdots) := P^\perp_{\dot{\gamma}} \widetilde{G_1} (\cdots) = \widetilde{G_1} (\cdots) - \langle \widetilde{G_1} (\cdots), \dot{\gamma}\rangle \dot{\gamma},
\end{align*}
which is obviously analytic again.
\end{proof}

\section{Proof of the main theorem by Cauchy's method of majorants}\label{sec:mainpart}

In order to prove the main result, Theorem \ref{main}, we deduce a recursive estimate for $\|\partial^l \gamma \|_{H^1 }$ by putting together the results of the previous sections.  We then apply Cauchy's method of majorants to the recursive formula to achieve the analyticity of $\gamma$.

Let $m := 1 > \frac{1}{2}$ and $\gamma :\RZ\rightarrow \R^n$, $\gamma = (\gamma_1,\ldots, \gamma_n)$, be a simple, closed, and by arc-length parametrized curve with $\gamma \in C^\infty (\RZ,\R^n)$ and $\gamma$ a critical point of the Möbius energy $\E$.  Then (\ref{formula:deltaE}) tells us that $\gamma$ solves the Euler-Lagrange equation
\begin{align*}
( H\gamma,h)_{L^2([0,1], \R^n)} = 0
\end{align*}
for all $h \in H^2 (\RZ,\R^n)$ and therefore
\begin{align}\label{formula:His0}
H\gamma = P^\perp_{\dot{\gamma}} Q\gamma + P^\perp_{\dot{\gamma}} R_1\gamma  +  P^\perp_{\dot{\gamma}} R_2
\gamma  \equiv 0
\end{align}
on $\RZ$ by (\ref{formula:H}) and (\ref{formula:Htilde}). Let $l\in\N$. Then by Corollary \ref{kor:Ql}, (\ref{formula:His0}), the definition of the orthogonal projection, and the triangle inequality
\begin{align}\label{formula:partiallgamma}
\begin{split}
\|\partial^{l+3} \gamma \|_{H^1 } & \leq \widetilde{C} \|\partial^l Q\gamma\|_{H^1} \\
& = \widetilde{C}\left(\left\| \partial^l \left( P^T_\gamma Q\gamma + P^\perp_{\dot{\gamma}} R_1\gamma  +  P^\perp_{\dot{\gamma}} R_2\gamma \right) \right\|_{H^1} \right) \\
& \leq \widetilde{C} \left(\|\partial^l P^T_\gamma Q\gamma\|_{H^1} + \|\partial^l P^\perp_{\dot{\gamma}} R_1\gamma \|_{H^1} + \|\partial^l P^\perp_{\dot{\gamma}} R_2\gamma \|_{H^1} \right).
\end{split}
\end{align}
For the next estimates we assume $\frac{1}{2} \geq \varepsilon > 0$. Then the estimates of the Sections \ref{sec:fractionalsobolevspaces} and \ref{sec:bilinearHilberttransform} imply the following estimate for the tangential part of $Q$.

\begin{lem}\label{est:Q}
	There exists a positive constant $C_Q <\infty$ which is independent of $\varepsilon$ and $l$ such that the following estimate holds:
	\begin{align*}
	\|\partial^l P^T_\gamma Q^\varepsilon\gamma\|_{H^1} \leq  C_Q \sum_{k_1=0}^l \sum_{k_2=0}^{k_1} \binom{l}{k_1} \binom{k_1}{k_2}  \|\partial^{l-k_1+2} \gamma \|_{H^1} \|\partial^{k_1-k_2+2} \gamma \|_{H^1} \|\partial^{k_2 +1}\gamma \|_{H^1}. 
	\end{align*}
\end{lem}
\begin{proof}
	Applying the Leibniz rule twice, we obtain 
	\begin{align*}
	\partial^l & \left(P^T_\gamma Q^\varepsilon\gamma (x)\right)\\
	= & \partial^l \left(4  \int_{|w|\in [\varepsilon, \frac{1}{2}]}  \iint_{[0,1]^2} (1-t) (-t) \frac{\left<\ddot{\gamma}(x+tw), \ddot{\gamma}(x+stw)\right>_{\R^n}}{w} \dot{\gamma}(x) \mathrm{d}s \mathrm{d}t \mathrm{d}w \right) \\
	= &  4  \int_{|w|\in [\varepsilon, \frac{1}{2}]}  \iint_{[0,1]^2} (1-t) (-t) \partial^l \left(\frac{\left<\ddot{\gamma}(x+tw), \ddot{\gamma}(x+stw)\right>_{\R^n}}{w} \dot{\gamma}(x)\right) \mathrm{d}s \mathrm{d}t \mathrm{d}w \\
	= & 4  \int_{|w|\in [\varepsilon, \frac{1}{2}]} \iint_{[0,1]^2} (1-t) (-t) \\
	& \cdot \sum_{k_1=0}^l \sum_{k_2=0}^{k_1} \binom{l}{k_1} \binom{k_1}{k_2} \frac{\left<\partial^{l-k_1+2} \gamma(x+tw), \partial^{k_1-k_2+2} \gamma (x+stw)\right>_{\R^n}}{w} \partial^{k_2 +1}\gamma(x) \mathrm{d}s \mathrm{d}t \mathrm{d}w \\ 
	= &  4 \iint_{[0,1]^2} (1-t) (-t)  \sum_{k_1=0}^l \sum_{k_2=0}^{k_1} \binom{l}{k_1} \binom{k_1}{k_2}  \sum_{k=1}^n H_{t,st}^\varepsilon(\partial^{l-k_1+2} \gamma_k, \partial^{k_1-k_2+2} \gamma_k )(x) \partial^{k_2 +1}\gamma(x) \mathrm{d}s \mathrm{d}t
	\end{align*}
	where interchanging integrals and derivative are permitted due to the smoothness of the integrand. Hence, by using the Banach algebra property of $H^1$ (\ref{prop:Banachalgebra}), the estimate of the approximation of the bilinear Hilbert transform (Theorem \ref{thm:BilinearHT}) for each of the components, the equivalence of the Sobolev norms and the inequality of Cauchy-Schwarz we get component-wise for all $1\leq m \leq n$
	\begin{align*}
	& \|\partial^l (P^T_\gamma Q^\varepsilon\gamma)_m\|_{H^1} \\
	& \leq  4 \iint_{[0,1]^2} |1-t| |t|  \cdot \\
	& \quad \quad \quad \sum_{k_1=0}^l \sum_{k_2=0}^{k_1} \binom{l}{k_1} \binom{k_1}{k_2}  \sum_{k=1}^n C_{1} \| H_{t,st}^\varepsilon(\partial^{l-k_1+2} \gamma_k, \partial^{k_1-k_2+2} \gamma_k )\|_{H^1} \|\partial^{k_2 +1}\gamma_m \|_{H^1} \mathrm{d}s \mathrm{d}t \\
	& \leq  4 C_{1} \iint_{[0,1]^2} |1-t| |t| \cdot \\
	& \quad \quad \quad \sum_{k_1=0}^l \sum_{k_2=0}^{k_1} \binom{l}{k_1} \binom{k_1}{k_2}  \left( \sum_{k=1}^n C_H\|\partial^{l-k_1+2} \gamma_k \|_{H^1} \|\partial^{k_1-k_2+2} \gamma_k \|_{H^1} \right) \|\partial^{k_2 +1}\gamma_m \|_{H^1}  \mathrm{d}s \mathrm{d}t  \\
	& \leq  4 C_{1} C_H \sum_{k_1=0}^l \sum_{k_2=0}^{k_1} \binom{l}{k_1} \binom{k_1}{k_2}  \|\partial^{l-k_1+2} \gamma \|_{H^1} \|\partial^{k_1-k_2+2} \gamma \|_{H^1} \|\partial^{k_2 +1}\gamma_m \|_{H^1} < \infty.
	\end{align*}
	We finally get by the definition of the $H^1$-norm and the inequality of Cauchy-Schwarz
	\begin{align*}
	& \|\partial^l P^T_\gamma Q^\varepsilon\gamma\|_{H^1} \\
	& \leq  4 C_{1} C_H \sqrt{n} \sum_{k_1=0}^l \sum_{k_2=0}^{k_1} \binom{l}{k_1} \binom{k_1}{k_2}  \|\partial^{l-k_1+2} \gamma \|_{H^1} \|\partial^{k_1-k_2+2} \gamma \|_{H^1} \|\partial^{k_2 +1}\gamma_m \|_{H^1} < \infty.
	\end{align*}
	where $C_Q:= 4 C_{1} C_H \sqrt{n} < \infty$ is a positive constant which is independent of $\varepsilon$ and $l$.
\end{proof}

Let us now estimate the remaining parts of the decomposition of the gradient by using the form of $P^\perp_{\dot{\gamma}} R_1^\varepsilon \gamma $ (respectively, $P^\perp_{\dot{\gamma}} R_2^\varepsilon \gamma $), Faà di Bruno's formula and the Banach algebra property of $H^1$.
\begin{lem}\label{est:R1}
	Let us consider the mapping $f : \RZ \rightarrow \R^n \setminus \{0\} \times \R^{3n}$, $f = (f_{1},\ldots, f_{4n})$,
	\begin{align*}
	f(x) := \begin{pmatrix}
	\dot{\gamma} (x) \\
	\ddot{\gamma}(x) \\
	\ddot{\gamma}(x)  \\
	\ddot{\gamma}(x)
	\end{pmatrix}.
	\end{align*}
	Then we obtain positive constants $C_{R_1}, r <\infty$ which are independent of $\varepsilon$ and $l$  such that 
	the universal polynomial $p_l^{(4n)}$ from the multivariate form of Faa di Bruno's formula (\ref{universalpoly}) satisfy
	\begin{align*}
	\begin{split}
	&\|\partial^l P^\perp_{\dot{\gamma}} R_1^\varepsilon \gamma \|_{H^1} \\
	& \quad \leq C_{R_1} p_l^{(4n)}\left(\left\{\frac{(|\alpha|+1)!}{r^{|\alpha|+1}} \right\}_{|\alpha| \leq l}, \{C_1\|f_{i}^{(j)}\|_{H^1} \}_{1\leq j \leq l, 1\leq i \leq 4n}\right)
	\end{split}
	\end{align*}
	where $C_1$ is the constant from the estimate (\ref{prop:Banachalgebra}) with $m=1$.
\end{lem}
\begin{proof}
	We start by recalling the formula of $P^\perp_{\dot{\gamma}} R_1^\varepsilon \gamma$ computed in the previous section,
	\begin{align*}
	(P^\perp_{\dot{\gamma}} R_1^\varepsilon \gamma )(x) = \int_{|w|\in [\varepsilon, \frac{1}{2}]}  \iiiint_{[0,1]^4}  (r_1-r_2)^2  G_1 (f_1(x)) \mathrm{d} \psi_1 \mathrm{d} \psi_2 \mathrm{d} r_1 \mathrm{d} r_2 \mathrm{d} w
	\end{align*}
	where $G_1 : \R^n \setminus \{0\} \times \R^{3n} \rightarrow \R^n$ is an analytic function as defined in (\ref{formula:G1}) and (\ref{formula:tildeG1}) and
	\begin{align*}
	f_1(x) := \begin{pmatrix}
	\int_0^1 \dot{\gamma} (x+sw) \mathrm{d}s \\
	\ddot{\gamma}(x+r_2w+(r_1-r_2)\psi_1 w) \\
	\ddot{\gamma}(x+r_2w+(r_1-r_2)\psi_2 w))  \\
	\int_0^1 \ddot{\gamma}(x+tw) \mathrm{d}t
	\end{pmatrix}.
	\end{align*}
	
	 By interchanging derivative and integral due to the smoothness of the integrand and by applying the generalized Faà di Bruno's formula as stated in (\ref{universalpoly}) we obtain component-wise for all $1 \leq k\leq 4n$
	\begin{align*}
	& \partial^l (P^\perp_{\dot{\gamma}} R_1^\varepsilon \gamma )_k(x) \\
	& = \partial^l \int_{|w|\in [\varepsilon, \frac{1}{2}]}  \iiiint_{[0,1]^4}  (r_1-r_2)^2  G_{1,k} (f_1(x)) \mathrm{d} \psi_1 \mathrm{d} \psi_2 \mathrm{d} r_1 \mathrm{d} r_2 \mathrm{d} w \\
	& = \int_{|w|\in [\varepsilon, \frac{1}{2}]}  \iiiint_{[0,1]^4}  (r_1-r_2)^2  \partial^l G_{1,k} (f_1(x)) \mathrm{d} \psi_1 \mathrm{d} \psi_2 \mathrm{d} r_1 \mathrm{d} r_2 \mathrm{d} w \\
	& = \int_{|w|\in [\varepsilon, \frac{1}{2}]}  \iiiint_{[0,1]^4}  (r_1-r_2)^2 p_l^{(4n)}\left(\{(\partial^\alpha G_1)_k(f_1(x)) \}_{|\alpha| \leq l}, \{f_{1,i}^{(j)}(x) \}_{1\leq j \leq l, 1\leq i \leq 4n}\right) \mathrm{d} \psi_1 \mathrm{d} \psi_2 \mathrm{d} r_1 \mathrm{d} r_2 \mathrm{d} w.
	\end{align*}
	
	Applying the $H^1$-norm on $\partial^l (P^\perp_{\dot{\gamma}} R_1^\varepsilon \gamma )$ component-wise, we yield for any $1 \leq k \leq 4n$ by the definition of the $H^1$-norm, triangle inequality and the Banach algebra property of $H^1$
	\begin{align}\label{formula:plpl*}
	\begin{split}
	&\|\partial^l (P^\perp_{\dot{\gamma}} R_1^\varepsilon \gamma )_k\|_{H^1}  \\
	& \leq \int_{|w|\in [\varepsilon, \frac{1}{2}]}  \iiiint_{[0,1]^4}  |r_1-r_2|^2 \left\|p_l^{(4n)}\left(\{(\partial^\alpha G_1)_k(f_1) \}_{|\alpha| \leq l}, \{f_{1,i}^{(j)} \}_{1\leq j \leq l, 1\leq i \leq 4n}\right) \right\|_{H^1} \mathrm{d} \psi_1 \mathrm{d} \psi_2 \mathrm{d} r_1 \mathrm{d} r_2 \mathrm{d} w \\
	& \leq  \int_{|w|\in [\varepsilon, \frac{1}{2}]}  \iiiint_{[0,1]^4}  |r_1-r_2|^2 \cdot \\
	& \quad \quad \quad \quad \quad \quad \quad \quad \quad \quad  p_l^{(4n)}\left(\{ \left\|(\partial^\alpha G_1)_k(f_1)\right\|_{H^1} \}_{|\alpha| \leq l}, \{C_1\|f_{1,i}^{(j)}\|_{H^1} \}_{1\leq j \leq l, 1\leq i \leq 4n}\right) \mathrm{d} \psi_1 \mathrm{d} \psi_2 \mathrm{d} r_1 \mathrm{d} r_2 \mathrm{d} w \\
	\end{split}
	\end{align}
	where $p_l^{(4n)}$ is the polynomial coming from (\ref{universalpoly}). The constant $C_1$ does not appear in front of $ \left\|(\partial^\alpha G_1)_k(f_1)\right\|_{H^1}$ because $p_l^{(4n)}$ is one-homogeneous in the first components.  
	
	Note that for any $1\leq j \leq l$  
	\begin{align*}
		\|f_{1,k}^{(j)}\|_{H^1} & \leq \|f_k^{(j)}\|_{H^1} \quad \forall \ 1\leq k\leq n, \ 3n+1\leq k \leq 4n, \\
		\|f_{1,k}^{(j)}\|_{H^1} & = \|f_k^{(j)}\|_{H^1} \quad \forall \ n+1\leq k\leq 3n,	
	\end{align*}
	and $\|f_k^{(j)}\|_{H^1}$ does not depend on any of the parameters $w,\psi_i, r_i, s, t$ for all $1\leq k\leq 4n$, $1\leq j \leq l$.
	Now, since $f_1$ is a continuous function on $\RZ$ and due to $\gamma \in C^\infty(\RZ,\R^n)$, $f_1([0,1])$ is a compact set $K$ of $\R^n\setminus \{0\} \times \R^{3n}$. By the analyticity of $G_1$, there hence are
	 constants $C_{G_1} < \infty, r >0$, such that 
	\begin{align}\label{formula:derofG1finally}
	\left|\partial^\alpha G_1 (x) \right| \leq C_{G_1} \frac{|\alpha|!}{r^{|\alpha|}}
	\end{align}
	for all $\alpha \in \N_0^{k_0}$, $k_0\in\N_0$, and $x\in K $. Without loss of generality, we assume $r < 1$.
	By defining the constant $C_{f_1}:= 1+\left\|f_1\right\|_{H^1} \leq 1+\left\|f\right\|_{H^1} < \infty$, using Lemma \ref{rem:compSobolevnorm} and (\ref{formula:derofG1finally}) we can estimate
	\begin{align*}
	& \left\|(\partial^\alpha G_1)_k(f_1)\right\|_{H^1} \\
	& = \left\|(\partial^\alpha G_{1,k}|_{K})(f_1)\right\|_{H^1} \\ 
	& \leq C_c \left\|\partial^\alpha G_{1,k}|_{K}\right\|_{C^1} (1+\left\|f_1\right\|_{H^1}) \\
	& \leq C_c C_{f_1} \left\|\partial^\alpha G_{1,k}|_{K}\right\|_{C^1} \\
	& = C_c C_{f_1} \max_{|\beta| = |\alpha|+1} \{\left\|\partial^\alpha G_{1,k}|_{K}\right\|_{\infty}, \left\|\partial^\beta G_{1,k}|_{K}\right\|_{\infty} \} \\
	& \leq C_c C_{f_1} C_{G_1} \max_{|\beta| = |\alpha|+1} \left\{\frac{|\alpha|!}{r^{|\alpha|}}, \frac{|\beta|!}{r^{|\beta|}} \right\} \\
	& \leq  C_c C_{f_1} C_{G_1} \frac{(|\alpha|+1)!}{r^{|\alpha|+1}} \max \left\{\frac{r^{|\alpha|+1}}{|\alpha|+ 1},1 \right\} 
	\end{align*}
Hence
	\begin{align*}
	\left\|(\partial^\alpha G_1)_k(f_1)\right\|_{H^1} & \leq \widetilde{C_{R_1}} \frac{(|\alpha|+1)!}{r^{|\alpha|+1}}
	\end{align*}
	for all $\alpha \in \N_0^{k_0}$, $k_0\in\N_0$, and $\widetilde{C_{R_1}}:=  C_c C_{f_1} C_{G_1} \sup_{\alpha} \max \left\{\frac{r^{|\alpha|+1}}{|\alpha|+ 1} \right\}< \infty$
	and with $C_{R_1} := \sqrt{4n} \widetilde{C_{R_1}} $ we get
	\begin{align*}
	\left\|(\partial^\alpha G_1)(f_1)\right\|_{H^1} & \leq C_{R_1}  \frac{(|\alpha|+1)!}{r^{|\alpha|+1}}
	\end{align*}
    In the end we arrive at
	\begin{align*}
	& \|\partial^l (P^\perp_{\dot{\gamma}} R_1^\varepsilon \gamma )\|_{H^1} \\
	& \quad \leq p_l^{(4n)}\left(\left\{ C_{R_1} \frac{(|\alpha|+1)!}{r^{|\alpha|+1}}\right\}_{|\alpha| \leq l}, \{\|C_1f_{i}^{(j)}\|_{H^1} \}_{1\leq j \leq l, 1\leq i \leq 4n}\right) 
	\end{align*}
	which is independent of $\varepsilon$. Using that $p_l^{(4n)}$ is one-homogeneous in the first components, we get the claim.
\end{proof}

The term $P^\perp_{\dot{\gamma}} R_2^\varepsilon \gamma$ can be estimated analogously:

\begin{lem}\label{est:R2}
	Let $l\in\N$ and $f = (\dot{\gamma}, \ddot{\gamma}, \ddot{\gamma}, \ddot{\gamma})$ as in Lemma \ref{est:R1}. Then there are positive constants $C_{R_2}, s <\infty$ which are independent of $\varepsilon$ and $l$ and the same universal polynomial $p_l^{(4n)}$ as in Lemma \ref{est:R1} such that 
	\begin{align*}
	&\|\partial^l P^\perp_{\dot{\gamma}} R_2^\varepsilon \gamma \|_{H^1} \leq  C_{R_2} p_l^{(4n)}\left(\left\{  \frac{(|\alpha|+1)!}{s^{|\alpha|+1}} \right\}_{|\alpha| \leq l}, \{C_1\|f_{i}^{(j)}\|_{H^1} \}_{1\leq j \leq l, 1\leq i \leq 4n}\right) .
	\end{align*} 
\end{lem}
\begin{proof}
	Completely similar to Lemma \ref{est:R1}.
\end{proof}

The estimates of the Lemmata \ref{est:Q}, \ref{est:R1} and \ref{est:R2} are also valid for the limit $\varepsilon \downarrow 0$ on the left-hand sides, what can be proved as below.

\begin{thm}\label{thm:estimates}
	Let $l\in\N$ and $f = (\dot{\gamma}, \ddot{\gamma}, \ddot{\gamma}, \ddot{\gamma})$ as in Lemma \ref{est:R1}. Then there are positive constants $C_Q, C_{R_1}, C_{R_2}, r, s <\infty$ which are independent of  $l$ and the same universal polynomial $p_l^{(4n)}$ as in Lemma \ref{est:R1} such that 
	
	\begin{align*}
	\|\partial^l P^T_\gamma Q\gamma\|_{H^1}  & \leq  C_Q \sum_{k_1=0}^l \sum_{k_2=0}^{k_1} \binom{l}{k_1} \binom{k_1}{k_2}  \|\partial^{l-k_1+2} \gamma \|_{H^1} \|\partial^{k_1-k_2+2} \gamma \|_{H^1} \|\partial^{k_2 +1}\gamma \|_{H^1}, \\
	\|\partial^l P^\perp_{\dot{\gamma}} R_1 \gamma \|_{H^1}  & \leq C_{R_1} p_l^{(4n)}\left(\left\{\frac{(|\alpha|+1)!}{r^{|\alpha|+1}} \right\}_{|\alpha| \leq l}, \{C_1\|f_{i}^{(j)}\|_{H^1} \}_{1\leq j \leq l, 1\leq i \leq 4n}\right),\\
	\|\partial^l P^\perp_{\dot{\gamma}} R_2 \gamma \|_{H^1} &  \leq  C_{R_2} p_l^{(4n)}\left(\left\{  \frac{(|\alpha|+1)!}{s^{|\alpha|+1}} \right\}_{|\alpha| \leq l}, \{C_1\|f_{i}^{(j)}\|_{H^1} \}_{1\leq j \leq l, 1\leq i \leq 4n}\right) . 
	\end{align*} 
\end{thm}
\begin{proof} 
	Let $l\in \N$ and $\{f_\varepsilon\}_{\varepsilon > 0}$ be a set of functions $f_\varepsilon : \RZ \rightarrow \R^n$ such that there exist positive constants $C_i<\infty$ with $\|\partial^i f_\varepsilon\|_{H^1} \leq C_i$ for all $\varepsilon > 0$ and $0\leq i \leq l$. Furthermore, let us assume that $\partial^i f_\varepsilon \overset{\varepsilon \downarrow 0}{\longrightarrow} f_i$ point-wise on $[0,1]$ with $f_i: \RZ \rightarrow \R^n$, $0\leq i \leq l$. Our aim is to show that $\partial^l f_\varepsilon \overset{\varepsilon \downarrow 0}{\longrightarrow} \partial^l f_0 \in H^1(\RZ,\R^n)$ weakly in $H^1$ and hence especially $$\|\partial^l f_0\|_{H^1} \leq \liminf_{\varepsilon\downarrow 0} \| \partial^l f_\varepsilon \|_{H^1} \leq C_l.$$ Then the statement of the theorem follows directly from applying this abstract argument to $f_\varepsilon = \partial^l P^T_\gamma Q^\varepsilon\gamma$, $f_\varepsilon =  \partial^l P^\perp_{\dot{\gamma}} R_1^\varepsilon \gamma$ and $f_\varepsilon =  \partial^l P^\perp_{\dot{\gamma}} R_2^\varepsilon \gamma$ and using the previous Lemmata \ref{est:Q}, \ref{est:R1} and \ref{est:R2}. \\
	Since $H^1 (\RZ,\R^n)$ is a Hilbert space, there exists a subsequence $\{f_{\varepsilon_k}\}_{k \in \N}$ of $\{f_\varepsilon\}_{\varepsilon > 0}$ which converges weakly to a unique function $f\in H^1(\RZ,\R^n)$ in $H^1 (\RZ,\R^n)$ and by the compact embedding $H^1(\RZ,\R^n)\subset \subset C (\RZ,\R^n)$ strongly to $f$ with respect to the uniform norm. Due to the point-wise convergence of $f_{\varepsilon_k} \overset{k\rightarrow \infty}{\longrightarrow} f_0$ and the uniform boundedness of $\{f_{\varepsilon_k}\}_{k \in \N}$, $\{f_{\varepsilon_k}\}_{k \in \N}$ converges weakly to $f_0$ in $C (\RZ,\R^n)$ as well. By the uniqueness of limits of weak convergent series we get $f=f_0$ and since every weakly convergent subsequence of $\{f_\varepsilon\}_{\varepsilon > 0}$ has the same limit, $\{f_\varepsilon\}_{\varepsilon > 0}$ converges weakly to $f_0\in H^1(\RZ,\R^n)$ with respect to the $H^1$-norm and strongly to $f_0$ with respect to the $L^2$norm. By the same argument, $\{\partial^l f_\varepsilon\}_{\varepsilon > 0}$ converges weakly to $f_l\in H^1(\RZ,\R^n)$ with respect to the $H^1$-norm and strongly to $f_l$ with respect to the uniform norm.\\	
	It remains to show $\partial^l f_0 = f_l$ which follows by
	\begin{align*}
	0 & = \lim_{\varepsilon\downarrow 0} \left(\int_0^1 \partial^l f_{\varepsilon} (x) \phi(x) \mathrm{d}x - (-1)^l \int_0^1 f_{\varepsilon} (x) \partial^l \phi (x) \mathrm{d} x  \right) \\
	& =  \int_0^1 f_l (x) \phi(x) \mathrm{d}x - (-1)^l \int_0^1 f_0 (x) \partial^l \phi (x) \mathrm{d} x
	\end{align*}
	for all $\phi\in C^\infty([0,1],\R^n)$ where we used Lebesgue's theorem. 	
\end{proof}

\begin{proof}[Proof of Theorem \ref{main}]
	Let $\gamma :\RZ\rightarrow \R^n$, $\gamma = (\gamma_1,\ldots, \gamma_n)$, be a closed, simple, and by arc-length parametrized curve in $ H^{\frac{3}{2}} (\RZ,\R^n)$ that is a critical point of the Möbius energy $\E$. We deduce by Theorem \ref{smooth} that $\gamma\in C^\infty(\RZ,\R^n)$. Moreover, we set $f:\RZ\rightarrow \R^{4n}$,
	\begin{align*}
	f(x):= \left(\begin{matrix}
	\dot{\gamma}(x)\\
	\ddot{\gamma}(x)\\
	\ddot{\gamma}(x)\\
	\ddot{\gamma}(x)
	\end{matrix}\right)
	\end{align*}
	and $a_l :=C_1  \|\partial^l f\|_{H^1}$ for all $l\in\N$. Our main goal is to determine positive constants $\delta,C_\gamma < \infty$ such that 
	\begin{align*}
	a_l \leq C_\gamma \frac{l!}{\delta^l} 
	\end{align*}
	from which we can directly conclude the analyticity of $\gamma$ on $\RZ$ by Corollary \ref{formula:partialfH}. \\
	As explained above in the beginning of the section, (\ref{formula:His0}) and therefore (\ref{formula:partiallgamma}) hold for any $l\in \N$, more precisely  	
	\begin{align*}
	\begin{split}
	\|\partial^{l+3} \gamma \|_{H^1 } \leq \widetilde{C}\left(\|\partial^l P^T_\gamma Q\gamma\|_{H^1} + \|\partial^l P^\perp_{\dot{\gamma}} R_1\gamma \|_{H^1} + \|\partial^l P^\perp_{\dot{\gamma}} R_2\gamma \|_{H^1} \right).
	\end{split}
	\end{align*}
	Due to the definition of the Sobolev norm and the previous Theorem \ref{thm:estimates} we may conclude that
	\begin{align*}	\begin{split}
	a_{l+1} = & \|\partial^{l+1} f\|_{H^1} \leq 3 \|\partial^{l+3} \gamma \|_{H^1} + \|\partial^{l+2} \gamma \|_{H^1} \\
	\leq & \ 3\widetilde{C} \left(\|\partial^l P^T_\gamma Q\gamma\|_{H^1} + \|\partial^l P^\perp_{\dot{\gamma}} R_1\gamma \|_{H^1} + \|\partial^l P^\perp_{\dot{\gamma}} R_2\gamma \|_{H^1} \right) + a_l\\
	\leq & \ 3\widetilde{C} \bigg( C_Q \sum_{k_1=0}^l \sum_{k_2=0}^{k_1} \binom{l}{k_1} \binom{k_1}{k_2}  \|\partial^{l-k_1+2} \gamma \|_{H^1} \|\partial^{k_1-k_2+2} \gamma \|_{H^1} \|\partial^{k_2 +1}\gamma \|_{H^1} \\
	& + C_{R_1} p_l^{(4n)}\left(\left\{ \frac{(|\alpha|+1)!}{r^{|\alpha|+1}} \right\}_{|\alpha| \leq l}, \{C_1\|f_{i}^{(j)}\|_{H^1} \}_{1\leq j \leq l, 1\leq i \leq 4n}\right) \\
	& +  C_{R_2} p_l^{(4n)}\left(\left\{\frac{(|\alpha|+1)!}{s^{|\alpha|+1}}  \right\}_{|\alpha| \leq l}, \{C_1 \|f_{i}^{(j)}\|_{H^1} \}_{1\leq j \leq l, 1\leq i \leq 4n}\right) \bigg) + a_l.
	\end{split}
	\end{align*}
	Hence, there are constants $C < \infty$ and $r_\gamma>0$ such that
	\begin{align}\label{formula:al+1}
	\begin{split}
	a_{l+1} \leq & \ C \bigg( \sum_{k_1=0}^l \sum_{k_2=0}^{k_1} \binom{l}{k_1} \binom{k_1}{k_2}  a_{l-k_1} a_{k_1-k_2} a_{k_2} \\
	& +p_l^{(4n)}\left(\left\{ \frac{(|\alpha|+1)!}{r_\gamma^{|\alpha|+1}}\right\}_{|\alpha| \leq l}, \{ a_j \}_{1\leq j \leq l}\right) \\
	&+ a_l.
	\end{split}
	\end{align}
	For any $y\in \R^{4n}$, $y = (y_1,\ldots,y_{4n})$, we define a majorant $G: \R^{4n} \rightarrow \R^{4n}$ component-wise as 
	\begin{align*}
	G_i(y) := C \left( y_i^3 + \frac{1 }{(1 +\frac{ 4na_0 -  (y_1 + \cdots + y_{4n})}{r_\gamma})^2}\right) + y_i,
	\end{align*}
	for all $1\leq i \leq 4n$, which is clearly analytic around $a_0 (1,\ldots, 1)$. We now consider a solution to the following initial value problem 
	\begin{align}\label{formula:ODE}
	\begin{split}
	\dot{c}(t) &= G( c(t)), \\
	c(0) &= a_0 (1,\ldots, 1).
	\end{split}
	\end{align}
	The Theorem \ref{CK} implies that (\ref{formula:ODE}) has  a unique solution $c$ in $ C^\infty(]-\varepsilon, \varepsilon[,\R^{4n})$ for some $\varepsilon >0$ which is analytic around $0$. So $c$ can be expressed by its Taylor series in a neighborhood of $0$:
	\begin{align*}
	c (t):= \sum_{k=0}^{\infty} \frac{\widetilde{a}_{k} (1,\ldots, 1)}{k!}  t^k 
	\end{align*}
	where $\widetilde{a}_{k} (1,\ldots, 1)= c^{(k)}(0)$. By Theorem \ref{thm:Tayloranalytic} we know that there exist positive constants $r_c,M < \infty$ such that 
	\begin{align}\label{formula:canalytic}
	|c^{(l)}(t)| \leq M \frac{l!}{r_c^l}
	\end{align}
	for all $t\in B_{r_c}(0)$ and $l\in\N_0$.
	
	For $1\leq i\leq 4n$ and $l\in\N$ we have by Faà di Brunos formula as stated in Remark \ref{rem:FaadiBruno}
	\begin{align*}
	\partial^l G_i(c(t)) = & \ C \bigg(\sum_{k_1=0}^l \sum_{k_2=0}^{k_1} \binom{l}{k_1} \binom{k_1}{k_2}   c_i^{(l-k_1)}(t)  c_i^{(k_1-k_2)}(t) c_i^{(k_2)}(t)  \\
	& +  p_l^{(4n)}\left(\left\{ \frac{(|\alpha|+1)!}{r_\gamma^{|\alpha|+1}} \frac{1}{\left(1 + \left( \frac{4na_0}{r_\gamma} -\frac{\left(c_1(t) + \cdots + c_{4n}(t)\right)}{r_\gamma}\right)\right)^{|\alpha|+2}} \right\}_{|\alpha| \leq l}, \{ c_i^{(j)}(t) \}_{1\leq j \leq l, 1 \leq i \leq 4n}\right)  \bigg) \\
	& + c_i^{(l)} (t),
	\end{align*}
	so we obtain by the initial condition (\ref{formula:ODE})
	\begin{align*}
	\partial^l G_i (c(0)) = & \ C \bigg(  \sum_{k_1=0}^l \sum_{k_2=0}^{k_1} \binom{l}{k_1} \binom{k_1}{k_2}   \tilde a_{l-k_1} \tilde a_{k_1-k_2} \tilde a_{k_2} \\
	& + p_l^{(4n)}\left(\left\{ \frac{(|\alpha|+1)!}{r_\gamma^{|\alpha|+1}} \right\}_{|\alpha| \leq k}, \{\tilde a_j \}_{1\leq j \leq l}\right)  \bigg) + \tilde a_l.
	\end{align*}
	Hence,
	\begin{align}\label{formula:tildeal+1}
	\begin{split}
	\tilde a_{l+1} = & \ C \bigg(  \sum_{k_1=0}^l \sum_{k_2=0}^{k_1} \binom{l}{k_1} \binom{k_1}{k_2}  \tilde a_{l-k_1} \tilde a_{k_1-k_2}  \tilde a_{k_2} \\
	& +p_l^{(4n)}\left(\left\{\frac{(|\alpha|+1)!}{r^{|\alpha|+1}}\right\}_{|\alpha| \leq l}, \{ \tilde a_j \}_{1\leq j \leq l}\right)  \bigg) + \tilde a_l.
	\end{split}
	\end{align}
Comparing (\ref{formula:al+1}) and (\ref{formula:tildeal+1}) and using the fact that the coefficients of the polynomial $p_l$ are non-negative, we get by $a_0 = \widetilde{a}_0 $ and induction
	\begin{align*}
	a_l  \leq \widetilde{a}_{l}. 
	\end{align*}
Hence $\|\partial^{l+1} \gamma \|_{H^1} \leq a_l \leq M\frac {l!}{r_C^l}$ for all $l\in\N_0$  by \eqref{formula:canalytic}. Thus $\gamma$ is analytic by Corollary \ref{formula:partialfH}.
\end{proof}


\begin{thebibliography}{DNPV12}

\bibitem[AF03]{AF03}
Robert~A. Adams and John J.~F. Fournier.
\newblock {\em Sobolev spaces}, volume 140 of {\em Pure and Applied Mathematics
  (Amsterdam)}.
\newblock Elsevier/Academic Press, Amsterdam, second edition, 2003.

\bibitem[AS97]{AS97}
David Auckly and Lorenzo Sadun.
\newblock A family of {M}\"obius invariant {$2$}-knot energies.
\newblock In {\em Geometric topology ({A}thens, {GA}, 1993)}, volume~2 of {\em
  AMS/IP Stud. Adv. Math.}, pages 235--258. Amer. Math. Soc., Providence, RI,
  1997.

\bibitem[Ber04]{B04}
S.~Bernstein.
\newblock Sur la nature analytique des solutions des \'equations aux
  d\'eriv\'ees partielles du second ordre.
\newblock {\em Math. Ann.}, 59(1-2):20--76, 1904.

\bibitem[Bla12]{B12}
Simon Blatt.
\newblock Boundedness and regularizing effects of {O}'{H}ara's knot energies.
\newblock {\em J. Knot Theory Ramifications}, 21(1):1250010, 9, 2012.

\bibitem[BR13]{BR13}
Simon Blatt and Philipp Reiter.
\newblock Stationary points of {O}'{H}ara's knot energies.
\newblock {\em Manuscripta Math.}, 140(1-2):29--50, 2013.

\bibitem[BRS16]{BRS16}
Simon Blatt, Philipp Reiter, and Armin Schikorra.
\newblock Harmonic analysis meets critical knots. {C}ritical points of the
  {M}\"obius energy are smooth.
\newblock {\em Trans. Amer. Math. Soc.}, 368(9):6391--6438, 2016.

\bibitem[CKS98]{CKS98}
J.~{Cantarella}, R.~B. {Kusner}, and J.~M. {Sullivan}.
\newblock Tight knot values deviate from linear relations.
\newblock {\em Nature}, 392:237--238, March 1998.

\bibitem[DNPV12]{DNPV12}
Eleonora Di~Nezza, Giampiero Palatucci, and Enrico Valdinoci.
\newblock Hitchhiker's guide to the fractional {S}obolev spaces.
\newblock {\em Bull. Sci. Math.}, 136(5):521--573, 2012.

\bibitem[Eva10]{E98}
Lawrence~C. Evans.
\newblock {\em Partial differential equations}, volume~19 of {\em Graduate
  Studies in Mathematics}.
\newblock American Mathematical Society, Providence, RI, second edition, 2010.

\bibitem[FHW94]{FHW94}
Michael~H. Freedman, Zheng-Xu He, and Zhenghan Wang.
\newblock M\"obius energy of knots and unknots.
\newblock {\em Ann. of Math. (2)}, 139(1):1--50, 1994.

\bibitem[Fol95]{F95}
Gerald~B. Folland.
\newblock {\em Introduction to partial differential equations}.
\newblock Princeton University Press, Princeton, NJ, second edition, 1995.

\bibitem[Fri58]{F58}
Avner Friedman.
\newblock On the regularity of the solutions of nonlinear elliptic and
  parabolic systems of partial differential equations.
\newblock {\em J. Math. Mech.}, 7:43--59, 1958.

\bibitem[Fuk88]{F88}
Shinji Fukuhara.
\newblock Energy of a knot.
\newblock In {\em A f\^ete of topology}, pages 443--451. Academic Press,
  Boston, MA, 1988.

\bibitem[Gev18]{G18}
Maurice Gevrey.
\newblock Sur la nature analytique des solutions des \'equations aux
  d\'eriv\'ees partielles. {P}remier m\'emoire.
\newblock {\em Ann. Sci. \'Ecole Norm. Sup. (3)}, 35:129--190, 1918.

\bibitem[GM99]{GM99}
Oscar Gonzalez and John~H. Maddocks.
\newblock Global curvature, thickness, and the ideal shapes of knots.
\newblock {\em Proc. Natl. Acad. Sci. USA}, 96(9):4769--4773, 1999.

\bibitem[Has06]{H06}
Yoshiaki Hashimoto.
\newblock A remark on the analyticity of the solutions for non-linear elliptic
  partial differential equations.
\newblock {\em Tokyo J. Math.}, 29(2):271--281, 2006.

\bibitem[He00]{H00}
Zheng-Xu He.
\newblock The {E}uler-{L}agrange equation and heat flow for the {M}\"obius
  energy.
\newblock {\em Comm. Pure Appl. Math.}, 53(4):399--431, 2000.

\bibitem[Hop32]{H32}
Eberhard Hopf.
\newblock \"uber den funktionalen, insbesondere den analytischen {C}harakter
  der {L}\"osungen elliptischer {D}ifferentialgleichungen zweiter {O}rdnung.
\newblock {\em Math. Z.}, 34(1):194--233, 1932.

\bibitem[HR79]{HR79}
Edwin Hewitt and Kenneth~A. Ross.
\newblock {\em Abstract harmonic analysis. {V}ol. {I}}, volume 115 of {\em
  Grundlehren der Mathematischen Wissenschaften [Fundamental Principles of
  Mathematical Sciences]}.
\newblock Springer-Verlag, Berlin-New York, second edition, 1979.
\newblock Structure of topological groups, integration theory, group
  representations.

\bibitem[IN14]{IN14}
Aya Ishizeki and Takeyuki Nagasawa.
\newblock A decomposition theorem of the {M}\"obius energy {I}: {D}ecomposition
  and {M}\"obius invariance.
\newblock {\em Kodai Math. J.}, 37(3):737--754, 2014.

\bibitem[IN16]{IN16}
Aya Ishizeki and Takeyuki Nagasawa.
\newblock The invariance of decomposed {M}\"obius energies under inversions
  with center on curves.
\newblock {\em J. Knot Theory Ramifications}, 25(2):1650009, 22, 2016.

\bibitem[Kat96]{K96}
Keiichi Kato.
\newblock New idea for proof of analyticity of solutions to analytic nonlinear
  elliptic equations.
\newblock {\em SUT J. Math.}, 32(2):157--161, 1996.

\bibitem[Kau05]{K05}
Louis~H. Kauffman.
\newblock The mathematics and physics of knots.
\newblock {\em Rep. Progr. Phys.}, 68(12):2829--2857, 2005.

\bibitem[KP02]{KP02}
Steven~G. Krantz and Harold~R. Parks.
\newblock {\em A primer of real analytic functions}.
\newblock Birkh\"auser Advanced Texts: Basler Lehrb\"ucher. [Birkh\"auser
  Advanced Texts: Basel Textbooks]. Birkh\"auser Boston, Inc., Boston, MA,
  second edition, 2002.

\bibitem[KS98]{KS98}
R.~B. Kusner and J.~M. Sullivan.
\newblock M\"obius-invariant knot energies.
\newblock In {\em Ideal knots}, volume~19 of {\em Ser. Knots Everything}, pages
  315--352. World Sci. Publ., River Edge, NJ, 1998.

\bibitem[Lew29]{L29}
Hans Lewy.
\newblock Neuer {B}eweis des analytischen {C}harakters der {L}\"osungen
  elliptischer {D}ifferentialgleichungen.
\newblock {\em Math. Ann.}, 101(1):609--619, 1929.

\bibitem[Mis00]{M00}
Rumen~L. Mishkov.
\newblock Generalization of the formula of {F}aa di {B}runo for a composite
  function with a vector argument.
\newblock {\em Int. J. Math. Math. Sci.}, 24(7):481--491, 2000.

\bibitem[MN57]{MN57}
C.~B. Morrey, Jr. and L.~Nirenberg.
\newblock On the analyticity of the solutions of linear elliptic systems of
  partial differential equations.
\newblock {\em Comm. Pure Appl. Math.}, 10:271--290, 1957.

\bibitem[Mor58a]{M58-1}
Charles~B. Morrey, Jr.
\newblock On the analyticity of the solutions of analytic non-linear elliptic
  systems of partial differential equations. {I}. {A}nalyticity in the
  interior.
\newblock {\em Amer. J. Math.}, 80:198--218, 1958.

\bibitem[Mor58b]{M58-2}
Charles~B. Morrey, Jr.
\newblock On the analyticity of the solutions of analytic non-linear elliptic
  systems of partial differential equations. {II}. {A}nalyticity at the
  boundary.
\newblock {\em Amer. J. Math.}, 80:219--237, 1958.

\bibitem[O'H91]{OH91}
Jun O'Hara.
\newblock Energy of a knot.
\newblock {\em Topology}, 30(2):241--247, 1991.

\bibitem[O'H92]{OH92}
Jun O'Hara.
\newblock Family of energy functionals of knots.
\newblock {\em Topology Appl.}, 48(2):147--161, 1992.

\bibitem[O'H94]{OH94}
Jun O'Hara.
\newblock Energy functionals of knots. {II}.
\newblock {\em Topology Appl.}, 56(1):45--61, 1994.

\bibitem[O'H03]{OH03}
Jun O'Hara.
\newblock {\em Energy of knots and conformal geometry}, volume~33 of {\em
  Series on Knots and Everything}.
\newblock World Scientific Publishing Co., Inc., River Edge, NJ, 2003.

\bibitem[Pet39]{P39}
I.~Petrosvky.
\newblock Sur l’analyticité des solutions des systèmes d’équations
  différentielles.
\newblock {\em Mat. Sbornik}, 5(47):3--70, 1939.

\bibitem[Rad26]{R26}
Tibor Rad\'o.
\newblock Das {H}ilbertsche {T}heorem \"uber den analytischen {C}harakter der
  {L}\"osungen der partiellen {D}ifferentialgleichungen zweiter {O}rdnung.
\newblock {\em Math. Z.}, 25(1):514--589, 1926.

\bibitem[Rei09]{R09}
Philipp Reiter.
\newblock {\em Repulsive Knot Energiesand Pseudodifferential Calculus: Rigorous
  Analysis and Regularity Theory for O’Hara’s Knot Energy Family
  $\E^{\alpha}$, $\alpha \in [2,3)$}.
\newblock PhD thesis, Technische Hochschule Aachen, 2009.

\bibitem[Rei12]{R12}
Philipp Reiter.
\newblock Repulsive knot energies and pseudodifferential calculus for
  {O}'{H}ara's knot energy family {$E^{(\alpha)},\alpha\in[2,3)$}.
\newblock {\em Math. Nachr.}, 285(7):889--913, 2012.

\bibitem[SvdM13]{SvdM13}
Pawe\l Strzelecki and Heiko von~der Mosel.
\newblock Menger curvature as a knot energy.
\newblock {\em Phys. Rep.}, 530(3):257--290, 2013.

\bibitem[Tay11]{T96}
Michael~E. Taylor.
\newblock {\em Partial differential equations {I}. {B}asic theory}, volume 115
  of {\em Applied Mathematical Sciences}.
\newblock Springer, New York, second edition, 2011.

\bibitem[TvdG96]{TvdG96}
J.~C. Turner and P.~van~de Griend, editors.
\newblock {\em History and science of knots}, volume~11 of {\em Series on Knots
  and Everything}.
\newblock World Scientific Publishing Co., Inc., River Edge, NJ, 1996.

\end{thebibliography}
\end{document}